\documentclass[12pt]{amsart}
\usepackage[final]{epsfig}
\usepackage{graphics}
\usepackage{amsmath}
\usepackage{amsfonts}
\usepackage{latexsym}
\usepackage{amssymb}
\usepackage{graphicx}
\usepackage{url}
\usepackage{epstopdf}
\usepackage{hyperref}

\usepackage{amsfonts}
\usepackage{amsthm}
\usepackage{a4wide}
\usepackage{marginnote}

\newtheorem{lemma}{Lemma}[section]
\newtheorem{proposition}[lemma]{Proposition}
\newtheorem{remark}[lemma]{Remark}

\newtheorem{theorem}{Theorem}

\newtheorem{corollary}[lemma]{Corollary}

\newcommand{\g}{{\gamma}}
\newcommand{\eps}{{\varepsilon}}
\newcommand{\C}{{\mathbb C}}

\newcommand{\R}{{\mathbb R}}
\newcommand{\Z}{{\mathbb Z}}
\newcommand{\RP}{{\mathbb {RP}}}

\renewcommand{\P}{\mathcal{P}}

\title{Introducing symplectic billiards}

\author{Peter Albers}
\author{Serge Tabachnikov}

\address{ Peter Albers\\
 Mathematisches Institut\\
 Ruprecht-Karls-Universit\"at Heidelberg\\
 Germany}
\email{palbers@{}mathi.uni-heidelberg.de}

\address{ Serge Tabachnikov\\
 Department of Mathematics\\
 Pennsylvania State University\\
 University Park, PA 16802\\
USA}
\email{tabachni@{}math.psu.edu}
\date{\today}

\begin{document}

\maketitle

\begin{abstract}
In this article we introduce a simple dynamical system called symplectic billiards. As opposed to usual/Birkhoff billiards, where length is the generating function, for symplectic billiards symplectic area is the generating function. We explore basic properties and exhibit several similarities, but also differences of symplectic billiards to Birkhoff billiards.
\end{abstract}

\section{Introduction} \label{intro}

Birkhoff billiard describes the motion of a free particle in a domain: when the particle hits the boundary, it reflects elastically so that the tangential component of its velocity remains the same and the normal component changes the sign. In the plane, this is the familiar law of geometric optics: the angle of incidence equals the angle of reflection, see Figure \ref{refl}.

 \begin{figure}[hbtp] 
\centering
\includegraphics[width=1.75in]{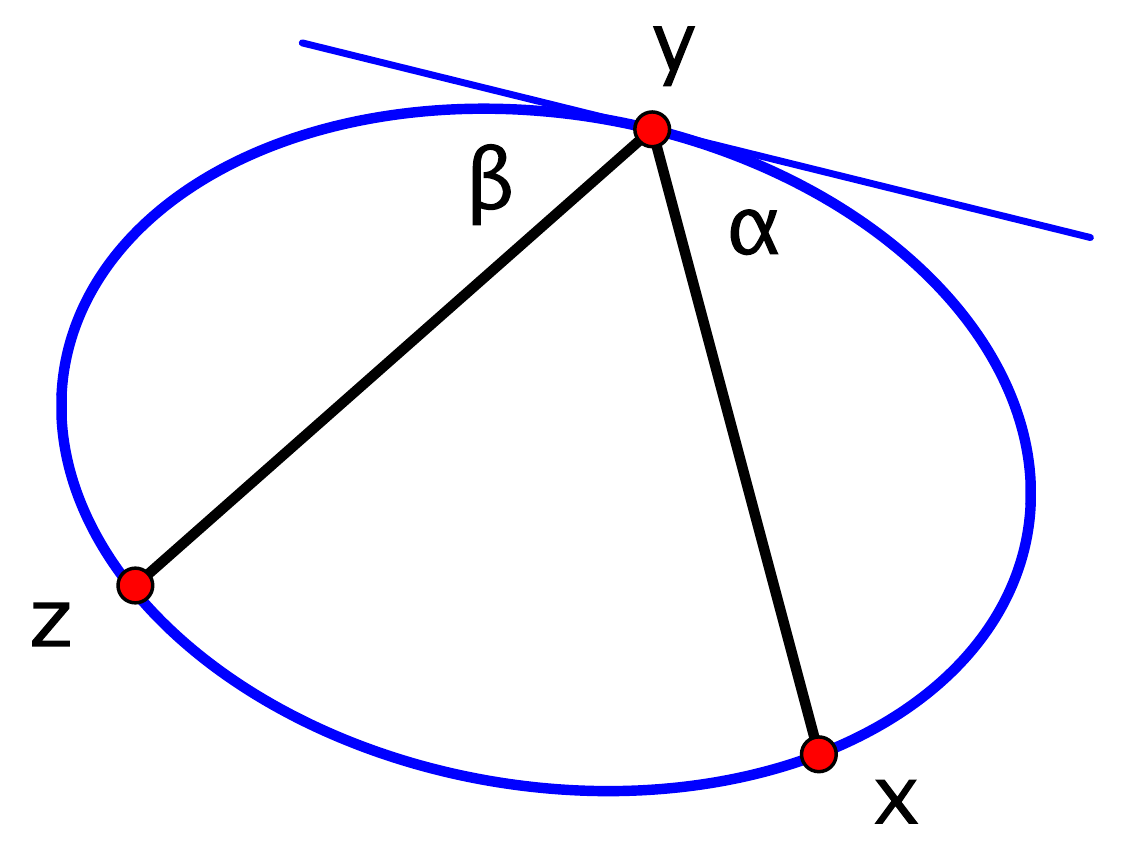}
\includegraphics[width=1.75in]{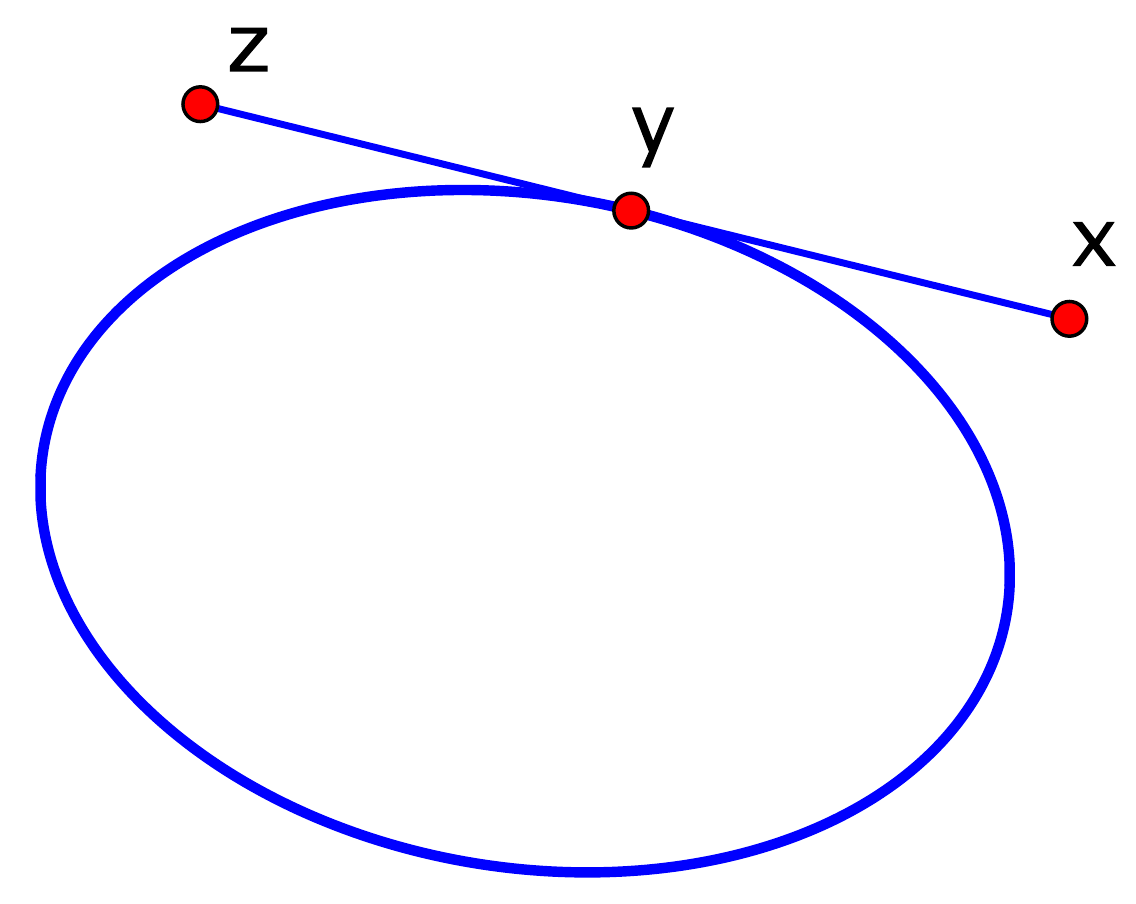}
\includegraphics[width=1.75in]{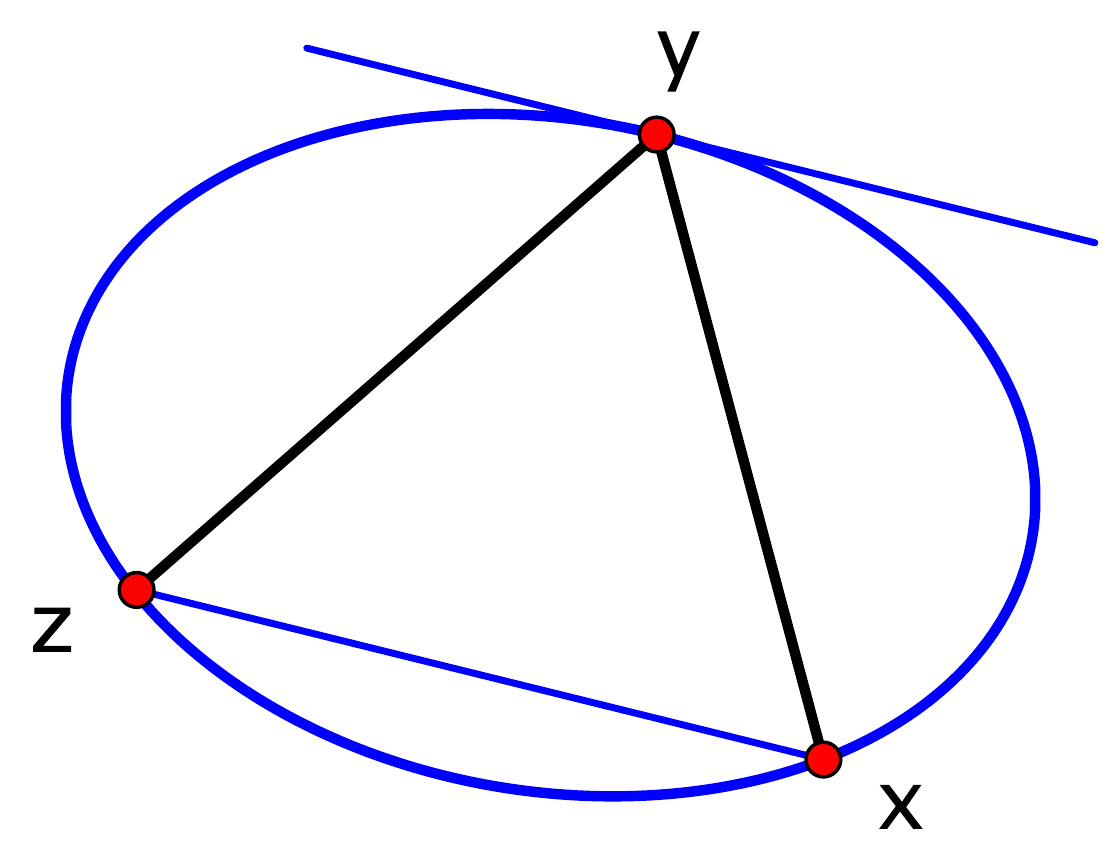}
\caption{Left, the usual billiard reflection: $xy$ reflects to $yz$ if $\alpha = \beta$. Middle, the outer billiard reflection: $x$ reflects to $z$ if $|xy|=|yz|$. Right, the symplectic billiard reflection: 
$xy$ reflects to $yz$ if $xz$ is parallel to the tangent line of the curve at point $y$.}
\label{refl}
\end{figure}

This reflection law has a variational formulation: if the points $x$ and $z$ are fixed, the position of the point $y$ on the billiard curve is determined by the condition that the length $|xy|+|yz|$ is extremal. In particular, periodic billiard trajectories are inscribed polygons of extremal perimeter length.

Another well-studied dynamical system is outer billiard, one such transformation in the exterior of a convex curve also depicted in Figure \ref{refl}. Outer billiard admits a variational formulation as well: periodic outer billiard orbits are circumscribed polygons of extremal area. 

It is natural to consider two other planar billiards: inner with area, and outer with the length as the generating functions. The variational formulation of the reflection depicted in Figure \ref{refl} on the right is as follows: if the points $x$ and $z$ are fixed, the position of the point $y$ on the billiard curve is determined by the condition that the area of the triangle $xyz$ is extremal. It would be natural to call this {\it area billiard} \footnote{A competing name would be {\it affine billiard} since this system commutes with affine transformations of the plane.}; due to higher-dimensional considerations, we prefer the term {\it symplectic billiard}.  

Analytically, the generating function of the symplectic billiard map is $\omega(x,y)$, where $\omega$ is the area form, as opposed to $|xy|$, the generating function of the usual billiard map. 

The aim of this paper is to introduce symplectic billiards and to make the first steps in their study. 
We believe that a system that has such a simple and natural definition should have interesting properties; in particular, one wants  to learn which familiar properties of the usual inner (Birkhoff) and outer billiards are specific to them, and which extend to the symplectic billiards and, perhaps, other similar systems. 

One can also define symplectic billiard in a convex domain with smooth boundary in a linear symplectic space $(\R^{2n},\omega)$. Let $M$ be its boundary. To define a reflection similarly to the planar case, one needs to choose a tangent direction at every point of $M$. This is canonically provided by the symplectic structure: the tangent line at $x\in M$ is the characteristic direction $\ker \omega|_{T_xM\times T_xM}$, that is, the kernel of the restriction of the symplectic form $\omega$ on the tangent hyperplane $T_x M$. (A similar idea is used in the definition of the multi-dimensional outer billiard, see, e.g., the survey \cite{DT}.) The generating function of this map is again $\omega(x,y)$, where $x,y \in M$.

This paper consists of three parts, the two-dimensional smooth and polygonal symplectic billiards, and multi-dimensional ones. Let us describe our main results.

In the planar case, we show that symplectic billiard is a monotone, area preserving twist map. In Theorem \ref{totar} we calculate the area of the phase space: it equals four times the area of the central symmetrization of the billiard table. 

Theorem \ref{Mathth} is a version of the well known theorem by Mather \cite{Ma}: if the billiard curve has a point of zero curvature, then the symplectic billiard possesses no caustics. Theorem \ref{Lazth} is a KAM theory result, an analog of Lazutkin's theorem \cite{Laz}: if the curve is sufficiently smooth and the curvature is everywhere positive, then the symplectic billiard has smooth caustics arbitrary close to the boundary. The coordinates in which the symplectic billiard map is a small perturbation of an integrable map are provided by the affine parameterization of the billiard curve. 

In Section \ref{ratcaust}, we apply the approach via exterior differential systems, introduced and implemented in the study of Birkhoff and outer billiards in \cite{BZ,BKNZ,GT,KZ,La,TZ,Tu}. It makes it possible to prove the existence of billiard tables that possess caustics consisting of periodic points. We remark that, for period 4, the billiard bounded by a Radon curve has this property. Another application of exterior differential systems is to small-period cases of the Ivrii conjecture which states that the set of periodic billiard orbits has zero measure (a weaker version says that this set has empty interior). Theorem \ref{Ivrth} asserts that  for planar symplectic billiards the set of 3-periodic points and that of 4-periodic points has empty interior.

In Section \ref{actspec}, we consider the area spectrum, that is, the set of areas of the polygons corresponding to the periodic orbits of the planar symplectic billiard. Let $A_n$ be the maximal area of an $n$-gon inscribed in the billiard curve. The theory of interpolating Hamiltonians \cite{Me,MM} implies an asymptotic series expansion of $A_n$ in negative even powers of $n$ whose first two terms are equal, up to numerical factors, to the area bounded by the billiard curve and its  affine length. Then the affine isoperimetric inequality implies that ellipses are uniquely determined by their area spectrum -- Theorem \ref{hear} (a similar result for outer billiards was obtained in \cite{Ta1}). 

We also show that symplectic billiards do not possess the finite blocking property (or, in different terminology, are insecure): for every pair of distinct points $x,y$ on the boundary curve $\g$ and every finite set $S$ inside $\gamma$, there exists a symplectic billiard trajectory from $x$ to $y$ that avoid $S$. See \cite{DG,Ta1} for insecurity of Birkhoff billiards.

Section \ref{polygons} concerns polygonal symplectic billiards. Theorems \ref{regpol} and \ref{trapezoid} assert that if the polygon is affine-regular or is a trapezoid, then all orbits are periodic, with explicitly described periods. In Propositions \ref{nbhd1} and \ref{nbhd2}, we describe two types of periodic orbits: those that appear in 2-parameter families (similarly to the periodic orbits of polygonal outer billiards), and the isolated, stable ones that survive small perturbations of the polygon. We do not know whether every convex polygon carries a periodic symplectic billiard orbit (this question is also open for the usual polygonal billiards,  in contrast with outer polygonal billiards where such a result is known).

After giving a careful definition of multi-dimensional symplectic billiard, we describe, in Section \ref{ellip}, the case of ellipsoids. As one may expect, this is a completely integrable case. Namely, in Theorem \ref{hodo}, we relate the symplectic billiard in an ellipsoid to the usual billiard in a different ellipsoid: given a trajectory of the symplectic billiard, construct a new polygonal line by connecting every second consecutive impact points (i.e, always skip one impact point); then a linear transformation (that depends only on the original ellipsoid) takes this polygonal line to a Birkhoff billiard trajectory on another ellipsoid (that again depends only on the original ellipsoid.) This result is analogous to a theorem of Moser and Veselov \cite{MV} that relates the discrete Neumann system with billiards in ellipsoids. We also explicitly describe the symplectic billiard dynamics inside a round sphere (Proposition \ref{superint}).

The last Section \ref{periodic} concerns periodic orbits in multi-dimensional symplectic billiards inside an arbitrary strictly convex closed smooth hypersurface. Theorem \ref{weak} asserts that, for every $k\ge 2$, there exist $k$-periodic trajectory. This is a weak result and, in Theorem \ref{smallper}, we obtain a stronger  one for small periods: 
the number of 3-periodic, and of 4-periodic, symplectic billiard orbits in $2n$-dimensional symplectic space is no less than $2n$ (the dihedral group $D_k$ acts on $k$-periodic trajectories by cyclic permutation of the vertices and reversing their order; what one counts are these $D_k$-orbits). 

We conclude this introduction with a remark concerning the interplay between convexity and symplectic geometry. Even though convexity is \textit{not} a symplectically invariant notion it has long been know that convex domains in symplectic space enjoy special rigidity properties. This culminated in Viterbo's conjecture \cite{Viterbo} which asserts an inequality between symplectic capacities and volume for \textit{convex} domains. It is well known that Viterbo's conjecture fails if the convexity assumption is dropped. And even though it has been proved in special cases in general, it is considered wide open. The fundamental nature of Viterbo's conjecture is demonstrated by the fact that it implies Mahler's conjecture from convex geometry \cite{AKO}. Recently, a renewed interested in this interplay arose, in part due to related results about systolic inequalities, see e.g., \cite{ABHS}.

Of course, the definition of the symplectic billiard map crucially relies on the convexity of the domain. At this point it is not more than idle speculation that this fact is more than a coincidence.

\bigskip 

{\bf Acknowledgments}. We are grateful to V.~Dragovic and B.~Jovanovic for a discussion of periodic  billiard orbits in ellipsoids, to R. Montgomery for a discussion of the Lexell theorem,  and to R. Schwartz for writing a computer program for experiments with polygons and for his insight concerning the symplectic billiards in trapezoids. In the past, the second author had numerous discussions of the topic of four planar billiards with E. Gutkin; his untimely death made it impossible for him to participate in this project. S.T. also acknowledges stimulating discussions of these ideas with S. Troubetzkoy. 

The first author is grateful to the hospitality of the Pennsylvania State University. He was supported by SFB/TRR 191. The second author is grateful to the hospitality of the Heidelberg University. He was supported by the NSF grant  DMS-1510055.

\section{Symplectic billiard in the plane} \label{planar}

\subsection{Symplectic billiard as an area preserving twist map} \label{refllaw}
\subsubsection{A precise look at the definition} \label{precise}
Let $\g$ be a smooth, strictly convex, closed, positively oriented curve, the boundary of our billiard table. 
Since $\g$ is strictly convex, for every point $x\in\g$ there exists a unique point $x^*\in\g$ with the property
\begin{equation*}
T_x\g=T_{x^*}\g\subset \R^2.
\end{equation*}
Clearly we have $x^{**}=x$. 

\begin{figure}[hbtp] 
\centering
\includegraphics[width=2.3in]{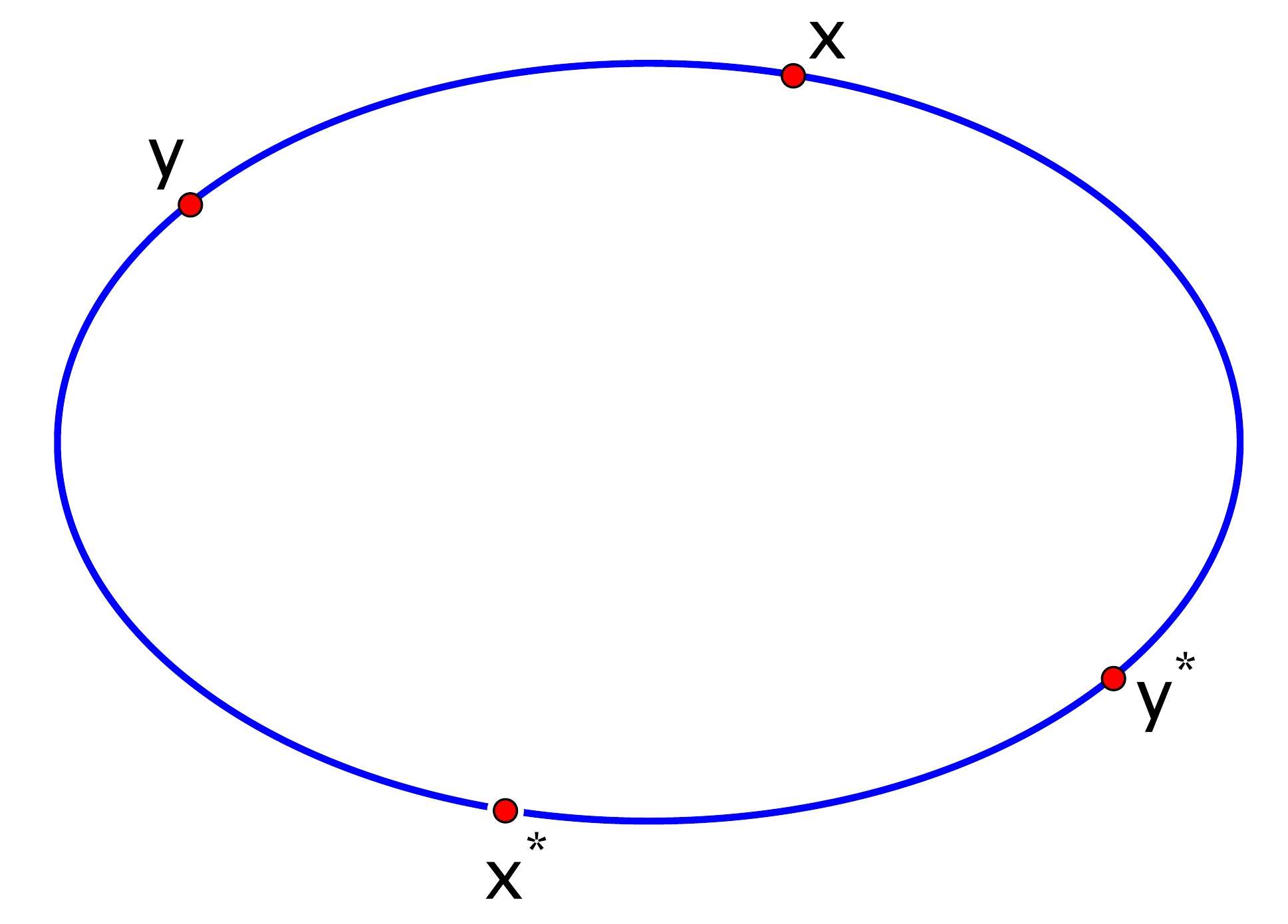}
\caption{Opposite points.}
\label{phase_space_2d}
\end{figure}
Denote by $\nu$ the outer normal of $\g$. Then
\begin{equation}\label{equal_tangent_space_as_om_on_normals}
T_x\g=T_y\g\quad \Longleftrightarrow \quad \omega(\nu_x,\nu_y)=0.
\end{equation}
Using the orientation of $\g$, we define
\begin{equation*}
\P:=\{(x,y)\in\g\times\g\mid x < y < x^*\}.
\end{equation*}
Note that this is actually antisymmetric in $x$ and $y$:
\begin{equation*}
\P=\{(x,y)\in\g\times\g\mid y^* < x <y\},
\end{equation*}
see Figure \ref{phase_space_2d}. Indeed,
\begin{equation*}
\P=\{(x,y)\in\g\times\g\mid \omega(\nu_x,\nu_y)>0\}.
\end{equation*}
We think of $\P$ as the (open, positive part of) phase space. The next lemma formalizes the definition of the symplectic billiard map.

\begin{lemma}\label{definition_of_Phi}
Given $(x,y)\in \P$, there exists a unique point $z\in\g$ with
\begin{equation*}
z-x\in T_y\g.
\end{equation*}
Moreover, the new pair $(y,z)$ lies again in our phase space: $(y,z)\in\P$.
\end{lemma}

\begin{proof}Since $(x,y)\in \P$, we know that $T_x\g\pitchfork T_y\g$. Using that $\g$ is convex, we obtain
\begin{equation}\label{intersection}
(x+T_y\g)\cap\g=\{x,z\}.
\end{equation}
Moreover, $z\neq x$ since otherwise $T_y\g=T_z\g=T_x\g$, which contradicts $x<y<x^*$. In other words, the line through $x$ parallel to $T_y\g$ intersects $\g$ in a new point $z$, see  Figure \ref{refl}. Equation \eqref{intersection} is equivalent to $z-x\in T_y\g$.

To show that $(y,z)\in \P$, we observe that if $y$ is close to $x$, then so is $z$, and therefore $\omega(\nu_x,\nu_y)>0$ implies $\omega(\nu_y,\nu_z)>0$. 

Now assume that $x<y<x^*$ and $\omega(\nu_y,\nu_z)\leq0$. By continuity and moving $y$ close to $x$, we can arrange $\omega(\nu_y,\nu_z)=0$. But then equation \eqref{equal_tangent_space_as_om_on_normals} implies that $T_y\g=T_z\g$, which implies $y=x$, a contradiction.
\end{proof}
Thus the map
\begin{equation*}
\Phi:\P\to\P,\ \  (x,y)\mapsto(y,z),
\end{equation*} 
with $z\in\g$ being the unique point satisfying $z-x\in T_y\g$, is well-defined.

\begin{remark}
{\rm 
Of course the same dynamics/map is defined on the negative part of phase space $\{(x,y)\in\gamma\times\gamma\mid x^* < y < x\}$ simply by reversing the orientation of $\g$.
}
\end{remark}

We extend $\Phi$ to the closure $\bar\P=\{(x,y)\in\gamma\times\gamma\mid x\leq y\leq x^*\}$ by continuity. The first case is obvious
\begin{equation}
\lim_{y\to x}\Phi(x,y)=(x,x),
\end{equation}
that is, the map extends as the identity. In the other case, we claim
\begin{equation}
\lim_{y\to x^*}\Phi(x,y)=(x^*,x).
\end{equation}
This follows from the observation that, due to convexity, $y\mapsto z(y)$ is monotone, and $\lim_{y\to x^*}T_y\g=T_{x^*}\g$. 

\begin{lemma}
The continuous extension $\Phi(x,x^*)=(x^*,x)$ is characterized by the 2-periodicity. That is,  $\Phi(x,y)=(y,x)$ is equivalent to $y\in\{x,x^*\}$.
\end{lemma}

\begin{proof}One direction is exactly the continuous extension. Now assume that $\Phi(x,y)=(y,x)$. If $(x,y)\in\P$ then, by Lemma \ref{definition_of_Phi},
\begin{equation*}
(x+T_y\g)\cap\g=\{x,y\}
\end{equation*}
with $x\neq y$. This is clearly a contradiction since $T_y\g\cap\g=\{y\}$. Therefore, either $(x,y)=(x,x)$ or $(x,y)=(x,x^*)$.
\end{proof}
\begin{remark} \label{css}
{\rm The envelope of the 1-parameter family of chords $x x^*$ is a caustic of our billiard. This envelope is called the centre symmetry set of $\g$, and it was studied in the framework of singularity theory \cite{GH}.
}
\end{remark}

Now we identify a generating function for the symplectic billiard map $\Phi$.

\begin{lemma}
The function 
\begin{equation*}
S:\P\to\R,\ \  (x,y)\mapsto S(x,y):=\omega(x,y)
\end{equation*} 
is a generating function for $\Phi$, that is
\begin{equation*}
\Phi(x,y)=(y,z) \quad \Longleftrightarrow \quad \frac{d}{dy}\Big[S(x,y)+S(y,z)\Big]=0.
\end{equation*}
\end{lemma}

\begin{proof}\begin{equation*}
\begin{aligned}
\frac{d}{dy}\Big[S(x,y)+S(y,z)\Big]=0 \quad &\Longleftrightarrow \quad \omega(x,v)+\omega(v,z)=0\;\;\forall v\in T_y\g\\
&\Longleftrightarrow \quad \omega(x-z,v)=0\;\;\forall v\in T_y\g\\
&\stackrel{(\ast)}{\Longleftrightarrow} \quad x-z\in T_y\g.\\
\end{aligned}
\end{equation*}
In $(\ast)$ we used $x\neq z$ which is due to $(x,y)\in \P$, compare with Lemma \ref{definition_of_Phi}.
\end{proof}

\begin{remark}
{\rm 
The area of the triangle $xyz$ on the right of Figure \ref{refl} equals
$$
\frac{1}{2} \left[ \omega(x,y) + \omega(y,z) + \omega(z,x) \right],
$$
which, ignoring the factor, differs from $S(x,y)+S(y,z)$ by a function of $x$ and $z$, having no effect on the partial derivative with respect to $y$.
}
\end{remark}

The symplectic billiard map commutes with affine transformations of the plane. Obviously, circles are completely integrable: concentric circles are the caustics; by affine equivariance, ellipses are completely integrable as well.

\subsubsection{Invariant area form and the twist condition} \label{Lagr}

Let $\g(t)$ be a parameterization of the curve $\g$ with $0\le x \le L$. Denote by $t^*$ the involution on the circle of parameters: the tangents at $\g(t)$ and $\g(t^*)$ are parallel. We have
$
S(t_1,t_2) = \omega(\g(t_1),\g(t_2)),
$
and $\Phi(t_1,t_2) = (t_2,t_3)$ if
\begin{equation} \label{Lagmap}
S_2 (t_1,t_2) + S_1(t_2,t_3) =0,
\end{equation}
where $S_{1}$ and $S_2$ denote the first partial derivatives with respect to the first and second arguments.

Following the theory of twist maps (see, e.g., \cite{Go}), we introduce new variables 
$$
s_1 = - S_1(t_1,t_2),\ s_2 = S_2(t_1,t_2),
$$
and a 2-form $\ \Omega = S_{12}(t_1,t_2)\ dt_1\wedge dt_2$ on $S^1\times S^1$.

\begin{lemma} \label{twist}
The variables $(t,s)$ are coordinates on the  phase space $\P$, and $\Omega$ is an area form therein.
The map $\Phi$ is a monotone twist map, and the area form is $\Phi$-invariant: $\Phi^*\Omega = \Omega$.
\end{lemma}

\begin{proof}Identify $\R^2$ with $\C$.
Then
$$
-\frac{\partial s_1}{\partial t_2}(t_1,t_2)= S_{12}(t_1,t_2) = \omega(\g'(t_1),\g'(t_2)) = \omega(-i\g'(t_1),-i\g'(t_2)) >0,
$$
since $-i \g'(t)$ is an outward normal vector to $\g$ at point $\g(t)$ and, in $\P$, one has by definition $\omega(\nu_{\g (t_1)},\nu_{\g(t_2)})>0$. It follows that $\Omega$ is an area form on $\P$.
It also follows that the Jacobian of the map $(t_1,t_2) \mapsto (t_1,s_1)$ does not vanish, therefore $(t_1,s_1)$ are coordinates.
The twist condition $\partial t_2/\partial s_1 < 0$ follows as well: indeed, $\partial s_1 / \partial t_2 = - S_{12} < 0$.

Finally, it follows that equation (\ref{Lagmap}) implies that the 2-form $S_{12} (t_1,t_2) dt_1\wedge dt_2$ is $\Phi$-invariant, indeed take the exterior derivative of  (\ref{Lagmap}) and wedge multiply by $d t_2$.  This completes the proof.
\end{proof}
Thus a variety of results about monotone twist maps apply to our symplectic billiard. In particular, we need to the following fact. For every period $n \ge 2$ and any rotation number $1\le k \le [n/2]$, the symplectic billiard has at least two distinct $n$-periodic orbits with rotation number $k$.

\subsubsection{Spherical and hyperbolic versions} \label{sphhyp}

Using the same generating function, the area of a triangle, one can define the symplectic billiard map in the  spherical and hyperbolic geometries. The definition is based on the Lexell theorem of spherical geometry, and its hyperbolic analog, that describes the locus of vertices of triangles with a given base and a fixed area, see \cite{MaMa,PaSu}. 

\begin{figure}[hbtp] 
\centering
\includegraphics[width=2.4in]{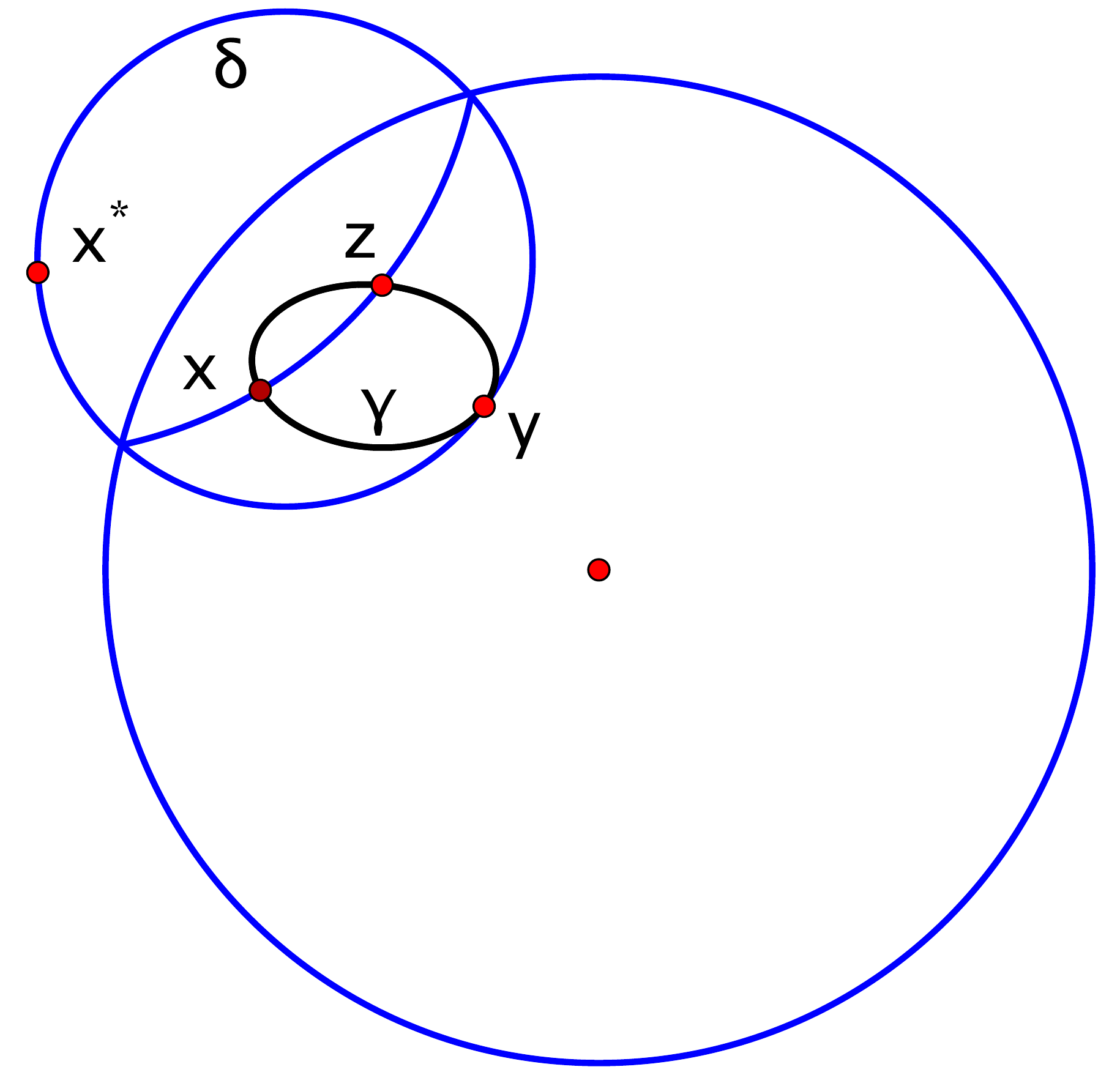}
\caption{The symplectic billiard map in the hyperbolic plane (Poincar\'e disk model): $\Phi(x,y)=(y,z)$.}
\label{hyper}
\end{figure}

Here is the formulation of Lexell's theorem. Let $xyz$ be a spherical triangle, and let $x^*$ and $z^*$ be the points antipodal to $x$ and $z$. Consider the arc of a spherical circle $x^* y z^*$. This arc is one half of the locus of points $w$ such that the area of the spherical triangle $xwz$ equals that of the triangle $xyz$, the other half being its reflection in the geodesic $xz$.

Consider the hyperbolic plane in the Poincar\'e disk model. Let now $x^*$ and $z^*$ be the inversions of points $x$ and $z$ in the unit circle. Then the above formulation holds without change (in hyperbolic terms, the arc $x^* y z^*$ is an equidistant curve, a curve of constant curvature less than 1).

This leads to the following definition of the symplectic billiard in the spherical and hyperbolic geometries. Let $\g$ be a closed convex curve, and let $x, y \in \g$. Let $\delta$ be the circle through point $x^*$ that is tangent to $\g$ at point $y$, and let $\delta^*$ be its image under the  antipodal involution, respectively, inversion in the absolute. Then  $\Phi(x,y) = (y,z)$,  where $z$ is the intersection point of $\delta^*$  with $\g$, different from $x$ (assuming that such a point is unique, otherwise the map $\Phi$ becomes multi-valued). See Figure \ref{hyper}.

\subsection{Total phase space area} \label{phasearea}

In this section we calculate the $\Omega$-area of the phase space.  

Let $D$ be the billiard table with the boundary $\g$. 
Parameterize $\g$ by the direction $\alpha$ of its tangent line. Choose an origin $O$ inside $D$, and let $p(\alpha)$ be the support function, that is, the distance from $O$ to the tangent line to $\g$ having direction $\alpha+\pi/2$, see Figure \ref{support}.

\begin{figure}[hbtp] 
\centering
\includegraphics[width=2in]{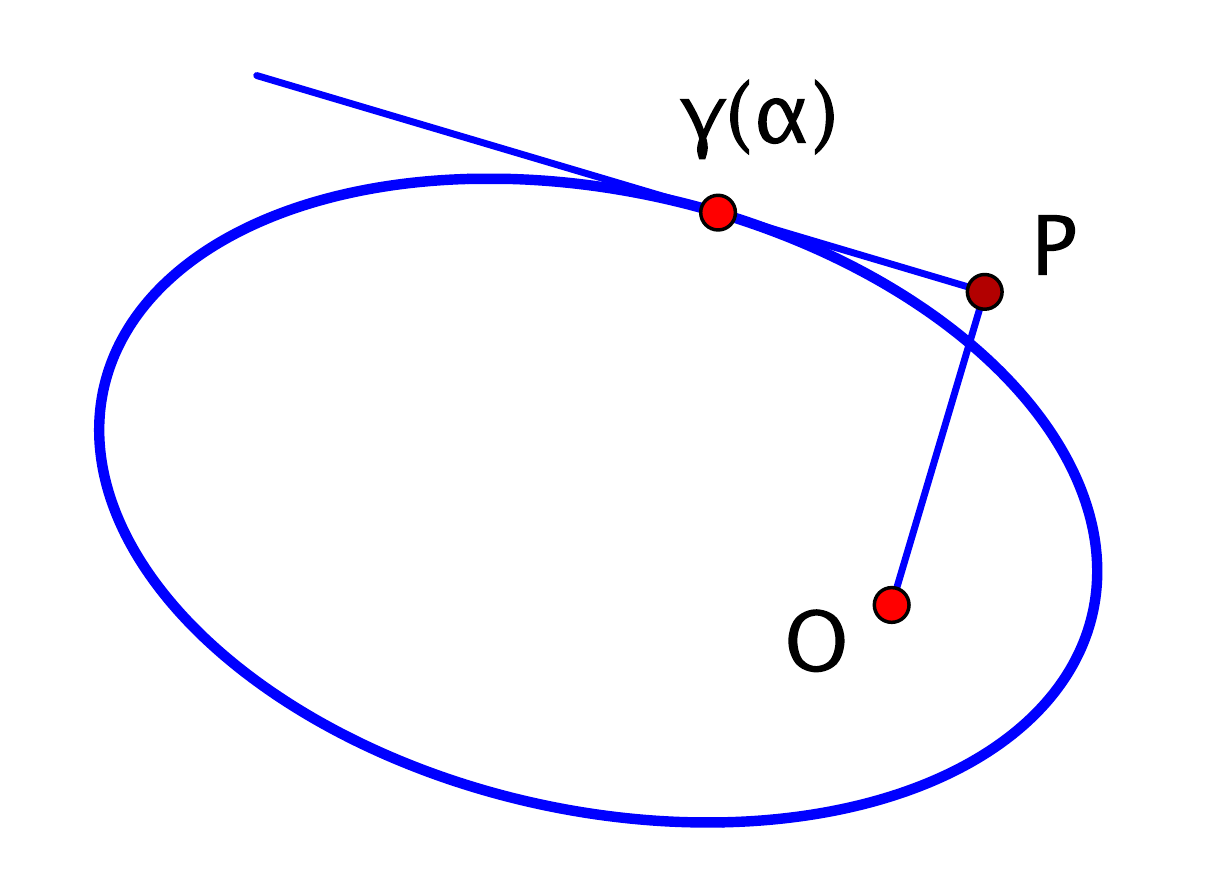}
\caption{The segment $OP$ has direction $\alpha$, and its length is $p(\alpha)$.}
\label{support}
\end{figure}

The Cartesian coordinates of the point $\g(\alpha)$ are given by the formulas
\begin{equation} \label{supform}
x(\alpha)=p(\alpha) \cos \alpha - p'(\alpha) \sin \alpha,\ y(\alpha)= p(\alpha) \sin \alpha + p'(\alpha) \cos \alpha.
\end{equation}
The perimeter length of $\g$ and the area bounded by it are given, respectively, by the integrals
$$
\int_0^{2\pi} p(\alpha)\ d\alpha \quad {\rm and} \quad \frac{1}{2} \int_0^{2\pi} [p(\alpha)+p''(\alpha)] p(\alpha)\ d\alpha,
$$
see, e.g., \cite{Sa}. 

The phase space $\P$ consists of pairs $\alpha_1,\alpha_2$ with $\alpha_1 < \alpha_2 < \alpha_1+\pi$. Let $w(\alpha)=p(\alpha)+p(\alpha+\pi)$ be the width of $D$ in the direction $\alpha$.

The  symmetrization $\bar D$ of the domain $D$  is the Minkowski sum with the centrally symmetric domain, scaled by 1/2. The support function $\bar p(\alpha)$ of $\bar D$ equals $[p(\alpha)+p(\alpha+\pi)]/2$. 

\begin{theorem} \label{totar}
The total $\Omega$-area of the phase space $\P$ equals four times the area of the symmetrization $\bar D$.
\end{theorem}

\begin{proof}Since $\Omega = S_{12}\ d\alpha_1\wedge d\alpha_2$, we need to calculate $S_{12} = [\g'(\alpha_1),\g'(\alpha_2)]$.

One has from (\ref{supform})
$$
\g'(\alpha) = [p(\alpha)+p''(\alpha)] (-\sin\alpha,\cos\alpha), 
$$
hence
$$
[\g'(\alpha_1),\g'(\alpha_2)] = [p(\alpha_1)+p''(\alpha_1)] [p(\alpha_2)+p''(\alpha_2)] \sin(\alpha_2-\alpha_1).
$$
Therefore the phase area is
$$
\int_0^{2\pi} \int_{\alpha_1}^{\alpha_1+\pi} [p(\alpha_1)+p''(\alpha_1)] [p(\alpha_2)+p''(\alpha_2)] \sin(\alpha_2-\alpha_1)\ d\alpha_2 d\alpha_1. 
$$
Consider the inner integral:
\begin{equation}\label{eqn:inner_integral}
\int_{\alpha_1}^{\alpha_1+\pi} p(\alpha_2) \sin(\alpha_2-\alpha_1)\ d\alpha_2 + \int_{\alpha_1}^{\alpha_1+\pi} p''(\alpha_2) \sin(\alpha_2-\alpha_1)\ d\alpha_2\;.
\end{equation}
We use integration by parts twice on the second summand to compute:
\begin{equation}\nonumber
\begin{aligned}
\int_{\alpha_1}^{\alpha_1+\pi} p''(\alpha_2)& \sin(\alpha_2-\alpha_1)\ d\alpha_2\\[2ex]
&=\underbrace{p'(\alpha_2) \sin(\alpha_2-\alpha_1)\biggr\rvert_{\alpha_1}^{\alpha_1+\pi}}_{=0}-\int_{\alpha_1}^{\alpha_1+\pi} p'(\alpha_2) \cos(\alpha_2-\alpha_1)\ d\alpha_2\\[2ex]
&=-\int_{\alpha_1}^{\alpha_1+\pi} p'(\alpha_2) \cos(\alpha_2-\alpha_1)\ d\alpha_2\\[2ex]
&=- p(\alpha_2) \cos(\alpha_2-\alpha_1)\biggr\rvert_{\alpha_1}^{\alpha_1+\pi} - \int_{\alpha_1}^{\alpha_1+\pi} p(\alpha_2) \sin(\alpha_2-\alpha_1)\ d\alpha_2\\[2ex]
&= p(\alpha_1) + p(\alpha_1+\pi)- \int_{\alpha_1}^{\alpha_1+\pi} p(\alpha_2) \sin(\alpha_2-\alpha_1)\ d\alpha_2\;.
\end{aligned}
\end{equation}
Thus the inner integral \eqref{eqn:inner_integral} evaluates to 
$
p(\alpha_1) + p(\alpha_1+\pi).
$

Now consider the outer integral and recall that the support function $\bar p(\alpha)$  of the symmetrization $\bar D$ equals $\tfrac12[p(\alpha)+p(\alpha+\pi)]$. 
\begin{equation}\nonumber
\begin{aligned}
\int_0^{2\pi} &[p(\alpha_1)+p''(\alpha_1)] [p(\alpha_1) + p(\alpha_1+\pi)]\ d\alpha_1 \\ 
&=\frac{1}{2} \int_0^{2\pi} [p(\alpha_1)+p''(\alpha_1) + p(\alpha_1+\pi)+p''(\alpha_1+\pi)] [p(\alpha_1) + p(\alpha_1+\pi)]\ d\alpha_1 \\
& =2 \int_0^{2\pi} [\bar p(\alpha) + \bar p''(\alpha)] \bar p(\alpha)\ d\alpha \\
&= 4 \;\text{area}(\bar D),
\end{aligned}
\end{equation}
as claimed.
\end{proof}
It is interesting to compare this with the usual billiards: in that case, the area of the phase space (with respect to the canonical area form on the space of oriented lines) equals twice the perimeter length of the boundary curve (see, e.g., \cite{Ta}).

\subsection{Existence and non-existence of caustics} \label{exnonex}

\subsubsection{Non-existence of caustics} \label{nonex}

The next result is a symplectic billiard version of Mather's theorem: If a smooth convex billiard curve has a point of zero curvature, then the (usual) billiard inside this curve has no caustics \cite{Ma} or \cite{MF}. 

\begin{theorem} \label{Mathth}
Let $\g$ be a smooth closed convex curve whose curvature vanishes at some point. Then the symplectic billiard in $\g$ has no caustics.
\end{theorem}

\begin{proof}According to Birkhoff's theorem, an invariant curve of an area preserving twist map is a graph of a function; see, e.g., \cite{KH}. 

Assume that our billiard has a caustic. Then one has a 1-parameter family of chords $x_1 x_2$ of the curve $\g$, corresponding to the points of the invariant curve. The graph property implies that if $\bar x_1 \bar x_2$ is a nearby chord from the same family and $\bar x_1$ has moved along $\g$ in the positive direction from $x_1$, then $\bar x_2$ also has moved in the positive direction from $x_2$. It follows that the chords intersect inside the curve $\g$ and, as a consequence, the caustic, which is the envelope of the lines containing these chords, lies inside the billiard table.

\begin{figure}[hbtp] 
\centering
\includegraphics[width=3.7in]{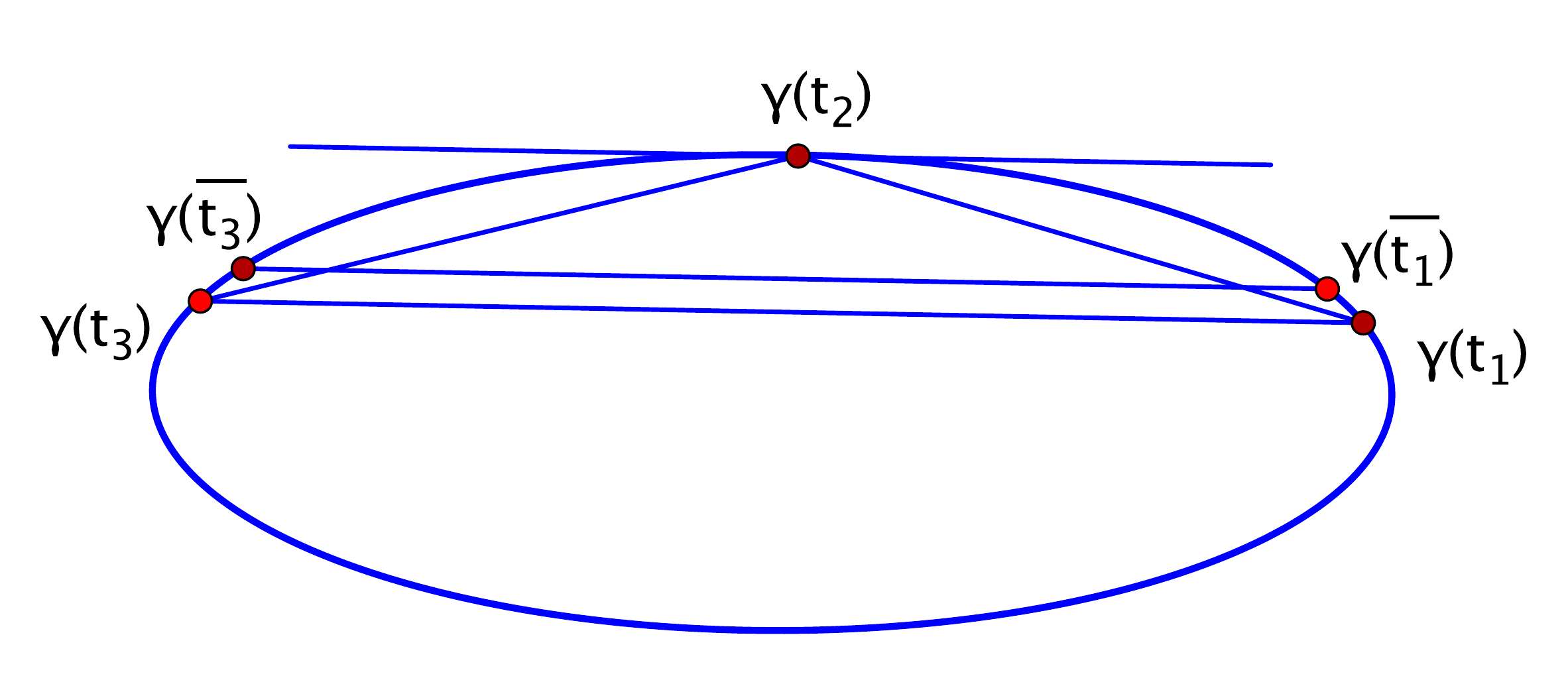}
\caption{The curvature at point $\g(t_2)$ vanishes.}
\label{curvature}
\end{figure}

Assume that a caustic exists, a chord $x_1 x_2$, tangent to the caustic, reflects to $x_2 x_3$, and the curvature at $x_2$ vanishes. Consider an infinitesimally close chord $\bar x_1 \bar x_2$, tangent to the same caustic. Since the curvature at $x_2$ vanishes, the tangent line at $\bar x_2$ is, in the linear approximation, the same as the one at $x_2$. Therefore, in the same linear approximation, the line $\bar x_1 \bar x_3$ is parallel to $x_1 x_3$, hence these lines do not intersect inside the billiard table. This contradiction proves the non-existence of the caustic.

Alternatively, one may use Mather's analytic condition from \cite{Ma}. This necessary condition for the existence of a caustic is $S_{22} (t_1,t_2) + S_{11} (t_2,t_3) < 0$.

Let $t$ be an arc length parameter on $\g$. Since $S(t_1,t_2) = \omega(\g(t_1),\g(t_2))$, we have
\begin{equation*}
S_{22} (t_1,t_2) = \omega(\g(t_1),\g''(t_2)) = k(t_2)\ \omega(\g(t_1),i \g'(t_2)),
\end{equation*}
and likewise for $S_{11} (t_2,t_3)$, 
where $k$ is the curvature of $\g$.  If $k(t_2)=0$, then $S_{22} (t_1,t_2) + S_{11} (t_2,t_3) =0$, violating Mather's  criterion.
\end{proof}
It would be interesting to find analogs of the results of Gutkin and Katok on the regions of the billiard table  free of caustics \cite{GK} and of Bialy  on the part of the phase space free of invariant curves \cite{Bia2}.

\subsubsection{Existence of caustics} \label{excaust}

The existence of caustics is provided by KAM theory, namely, Lazutkin's theorem \cite{Laz}, applied to symplectic billiards.  Let us make a simplifying assumption that the billiard curve $\g$ is infinitely smooth (this can be replaced by sufficiently high finite smoothness) and has everywhere positive curvature. 

\begin{theorem} \label{Lazth}
Arbitrary close to the curve $\g$, there exist smooth caustics for the symplectic billiard map; the union of these caustics has positive measure.
\end{theorem}

\begin{proof}It will be convenient  to use an affine parameterization of $\g$; let us recall the pertinent notions, e.g., \cite{Gug,LSZZ}.

We first give a geometric definition of the affine length of the curve followed by a more computational one.

Given $v\in T_x\gamma$, we consider a parametrization $\g(t)$ with $\g(0)=x$ and $\frac{d\g}{dt}(0)=v$.  Our aim is to define a cubic form $B(x,v)$ on $\g$. For this we denote by $A(\eps,t)$ the area bounded by the curve and the segment $\g(0) \g(\eps)$. We keep $t$ in the notation as a reminder for the parametrization, see below. This area is of third order in $\eps$, hence 
$$
B(x,v):= \lim_{\eps \to 0} \frac{A(\eps,t)}{\eps^3}
$$
is well-defined and, of course, independent of the parametrization (as long as $\g(0)=x$ and $\frac{d\g}{dt}(0)=v$.) 

That $B$ indeed is a cubic form is basically the chain rule. Indeed, replace $v$ by $av$ for some $a\neq0$. If we choose a new parametrization $\g(\tau)$ with $t=a\tau$ then at $t=\tau=0$
\begin{equation}\nonumber
\frac{d\g}{d\tau}=\frac{d\g}{dt}\,\frac{dt}{d\tau}=av.
\end{equation}
We now can define two enclosed areas with respect to these two parametrizations: $A(\eps, t)$ and $A(\eps,\tau)$. Of course, area is area, i.e.,
\begin{equation}\nonumber
A(\eps,t)=A(a\eps,\tau),
\end{equation}
and thus
\begin{equation}\nonumber
B(x,v)=\lim_{\eps \to 0} \frac{A(\eps, t)}{\eps^3}=\lim_{\eps \to 0} \frac{A(a\eps,\tau)}{(a\eps)^3}=a^{-3}\lim_{\eps \to 0} \frac{A(a\eps,\tau)}{\eps^3}=a^{-3}B(x,av).
\end{equation}
This confirms that $B$ is a cubic form. Therefore, $B^{1/3}$ is a 1-form, called an element of affine length. The integral of this form is called the affine length of the curve.

More conveniently for computations, a parameterization $\g(t)$ is affine if $[\g'(t),\g''(t)]=1$ for all $t$. In this case, $\g'''(t) = -k(t) \g'(t)$, where the positive function $k$ is called the affine curvature. The relation between the affine length parameter $t$ and the Euclidean arc length parameter $x$ is $dt = \kappa^{1/3} dx$, where $\kappa$ is the (Euclidean) curvature.

Consider an affine parameterization of the billiard curve. A chord $\g(t_1) \g(t_2)$ is characterized by the numbers $t_1$ and $\eps:=t_2-t_1$. We use $(t,\eps)$ as coordinates on the phase cylinder: $t$ is a cyclic coordinate, $\eps$ is non-negative and bounded above by some function of $t$ (since $t_1 \leq t_2 \leq t_1^*$).

Let us describe the billiard map in these coordinates. Let $\Phi: (t-\eps,\eps) \mapsto (t,\delta)$, where $\delta(t,\eps)$ is a function on the phase space. We claim that 
\begin{equation} \label{deltaeps}
\delta(t,\eps) = \eps + \eps^3 f(t,\eps),
\end{equation}
where $f$ is a smooth function. 

To prove this claim, let $\delta = a_0 + a_1 \eps + a_2 \eps^2 + a_3 \eps^3 + O(\eps^4)$, where $a_i$ are functions of $t$. Since the boundary $\eps=0$ consists of fixed points of the billiard map, $a_0=0$. By definition of the billiard reflection,
\begin{equation} \label{cross}
[\g(t+\delta)-\g(t-\eps),\gamma'(t)]=0,
\end{equation}
where the brackets denote the determinant made by two vectors.

Expand $\g(t+\delta)$ and $\g(t-\eps)$ in Taylor series up to 4th derivative and substitute to (\ref{cross}):
\begin{equation} \label{expand}
\left[(\delta+\eps)\g' + \left(\frac{\delta^2-\eps^2}{2}\right)\g'' + \left(\frac{\delta^3+\eps^3}{6}\right)\g''' + \left(\frac{\delta^4-\eps^4}{24}\right)\g^{''''},\g'\right]=0,
\end{equation}
where we suppress the argument $t$ from $\g(t)$ and its derivatives.

Since $[\g',\g'']=1$, the quadratic term in (\ref{expand}) yields $a_1=1$. Since $[\g',\g''']=0$, the cubic term does not give any information. Since $\g'''=-k\g$, one has $\g''''=-k'\g'-k\g''$, and hence $[\g'''',\g']=k \neq 0$. Therefore the quartic term in (\ref{expand}) yields $a_2=0$,  proving (\ref{deltaeps}). 

Formula (\ref{deltaeps}) implies that the billiard map is given by the formula:
$$
(t,\eps) \mapsto (t+\eps, \eps + \eps^3 g(t,\eps)),
$$
where $g$ is a smooth function. This is an area preserving map, a small perturbation of the integrable map $(t,\eps) \mapsto (t+\eps, \eps)$, satisfying the assumptions of  Lazutkin's theorem (Theorem 2 in \cite{Laz}). Applying this theorem concludes the proof.
\end{proof}
Figure \ref{soup} shows a computer generated phase portrait of the symplectic billiard. It looks like a typical area preserving twist map.

\begin{figure}[hbtp] 
\centering
\includegraphics[width=3.7in]{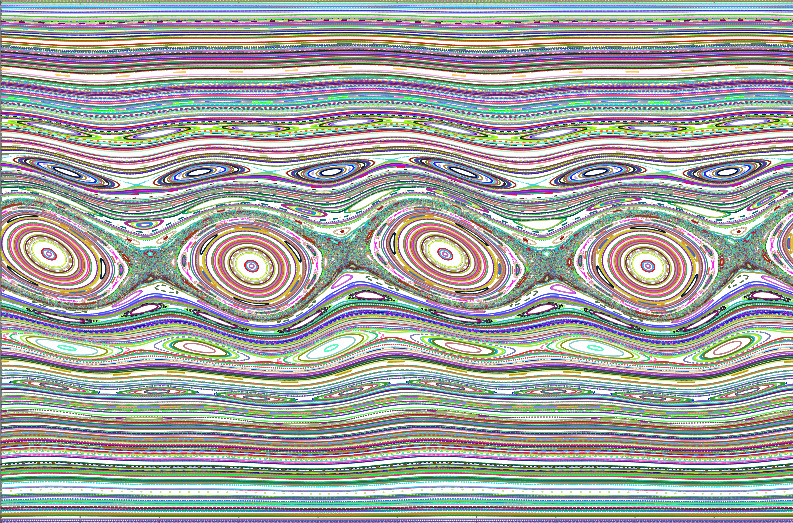}
\caption{Symplectic billiard inside a small perturbation of a circle. The invariant curve consisting of 2-periodic orbits is the lower horizontal line in the picture. We thank Michael Herrmann for providing the numerical simulation producing this picture.}
\label{soup}
\end{figure}

\begin{remark} \label{string}
{\rm A useful tool in the study of billiards is the string construction that recovers a billiard table from its caustic; it produces a 1-parameter family of tables. A similar, area, construction is known for outer billiards, see, e.g., \cite{DT}. In the context of symplectic billiards, we failed to discover a string construction.

}
\end{remark}

\subsection{Periodic caustics, Radon curves,  and $n=3,4$ versions of the Ivrii conjecture} \label{ratcaust}

\subsubsection{Distribution on the space of polygons} \label{distr}

In this section we apply the approach to periodic billiard trajectories via exterior differential systems, developed by Baryshnikov and Zharnitsky \cite{BZ} and, independently, by Landsberg \cite{La}, and applied to outer billiards in \cite{GT,TZ}; see also \cite{BKNZ}.

Suppose that the symplectic billiard has an invariant curve consisting of $n$-periodic points. A periodic point is a polygon $P=(z_1,z_2,\ldots,z_n)$ inscribed into the billiard curve $\g$ and having an extremal area. These $n$-gons form a 1-parameter family of inscribed polygons, connecting  $P=(z_1,z_2,\ldots,z_n)$ to the polygon $P'=(z_2,\ldots,z_n,z_1)$ with cyclically permuted vertices. The area remains constant in this  1-parameter family of area-extremal polygons.

Assume that $n\ge 3$. One can reverse the situation: start with an $n$-gon $P$ and try to construct a billiard curve $\g$ that has an invariant curve of $n$-periodic points including $P$. The main observation is that the directions of $\g$ at the vertices $z_i,\ i=1,\ldots,n$ are uniquely determined: they are the directions of the diagonals $z_{i+1} z_{i-1}$ (here and elsewhere the indices are taken in the cyclic order, so that $n+1=1$).

Consider the $2n$-dimensional space of $n$-gons in $\R^2$, and let  ${\mathcal P}_n$ be its open dense subset given by the condition that $z_{i+1}-z_{i-1}$ is not collinear with $z_{i+2}-z_{i}$ (in particular, each vector $z_{i+1}-z_{i-1}$ is non-zero).
Restricting the motion of the $i$-th vertex $z_i$ to the direction of the diagonal $z_{i+1} z_{i-1}$ defines an $n$-dimensional distribution ${\mathcal D}$ on ${\mathcal P}_n$ (its analog for the usual inner billiards was called in \cite {BZ} the Birkhoff distribution). An invariant curve of $n$-periodic points of the symplectic  billiard gives rise to a curve in ${\mathcal P}_n$, tangent to ${\mathcal D}$. We call curves tangent to $\mathcal{D}$ horizontal curves. We point out that, by the very definition of $\mathcal{P}_n$, we only consider invariant curves consisting of non-degenerate polygons, i.e., polygons in $\mathcal{P}_n$.

Let $A: {\mathcal P}_n \to \R$ be the algebraic area of a polygon given by
$$
A = \frac{1}{2} \sum_{i=1}^n \omega(z_i,z_{i+1}).
$$

\begin{theorem} \label{Bdistr}
The distribution ${\mathcal D}$ is tangent to the level hypersurfaces of the area function $A$. The distribution ${\mathcal D}$ is totally non-integrable on these level hypersurfaces: the tangent space at every point is generated by the vectors fields tangent to ${\mathcal D}$ and by their first commutators.
\end{theorem}

\begin{proof}
Let $z_i=(p_i,q_i)\in\R^{2n},\ i=1,\ldots,n$. Introduce the vector fields
\begin{equation} \label{vectf}
v_i = (p_{i+1}-p_{i-1}) \frac{\partial}{\partial p_i} + (q_{i+1}-q_{i-1}) \frac{\partial}{\partial q_i},
\end{equation}
where, as always, the indices are understood cyclically.
The fields $v_1,\ldots,v_n$ are linearly independent and they span ${\mathcal D}$ at every point. 

Geometrically, it is obvious that the fields $v_i$ preserve the area of a polygon. Analytically, 
$$
dA = \frac{1}{2} \sum_{i=1}^n (q_{i+1}-q_{i-1}) dp_i - (p_{i+1}-p_{i-1}) dq_i.
$$
Let 
$$
\theta_i = (q_{i+1}-q_{i-1}) dp_i - (p_{i+1}-p_{i-1}) dq_i,
$$
so that $dA = \sum \theta_i$. 
Then $\theta_i (v_j)=0$ for all $i,j$, and it follows that $dA$ vanishes on ${\mathcal D}$.
The common kernel of the 1-forms $\theta_i$ is the distribution ${\mathcal D}$.

Next, one has $[v_i,v_j]=0$ for $|i-j|\ge 2$, and 
$$
[v_{i+1},v_{i}] = (p_{i+2}-p_i)\frac{\partial}{\partial p_i} + (p_{i+1}-p_{i-1})\frac{\partial}{\partial p_{i+1}} +
(q_{i+2}-q_i)\frac{\partial}{\partial q_i} + (q_{i+1}-q_{i-1})\frac{\partial}{\partial q_{i+1}}.
$$
It follows that 
\begin{equation} \label{mat}
\theta_i([v_i,v_{i+1}]) = - \theta_{i+1} ([v_i,v_{i+1}]) = \det(z_{i+1}-z_{i-1},z_{i+2}-z_i) \neq 0,
\end{equation}
where the last inequality is due to the definition of ${\mathcal P}_n$. 

The distribution ${\mathcal D}$ has codimension $n-1$ on a constant-area hypersurface, and it suffices to show that the rank of the matrix
$\theta_i ([v_j,v_{j+1}]),\ i=1,\ldots,n-1, j=1,\ldots n$, is $n-1$. But it follows from (\ref{mat}) that the square submatrix with $i,j=1,\ldots,n-1$ is upper-triangular with non-zero diagonal entries. This concludes the proof.
\end{proof}
As in \cite{BZ}, one can draw two conclusions from Theorem \ref{Bdistr}. 

The first is that one has an extreme flexibility of deforming horizontal curves of the distribution ${\mathcal D}$, and hence of  billiard curves, still keeping an invariant curve consisting of $n$-periodic points (the technical statement is that there is a Hilbert manifold worth of such curves, obtained by deforming a circle and depending on functional parameters). We do not dwell on these technical issues; see \cite{BZ} for a detailed discussion in the usual billiard set-up or \cite{GT} for the case of outer billiards.  However, we present one concrete example for $n=4$.

\subsubsection{Radon curves} \label{Radonc}

 A centrally symmetric convex closed planar curve $\g$ is called a Radon curve if it has the property depicted in Figure \ref{Radon1}. 
 
\begin{figure}[hbtp] 
\centering
\includegraphics[width=2in]{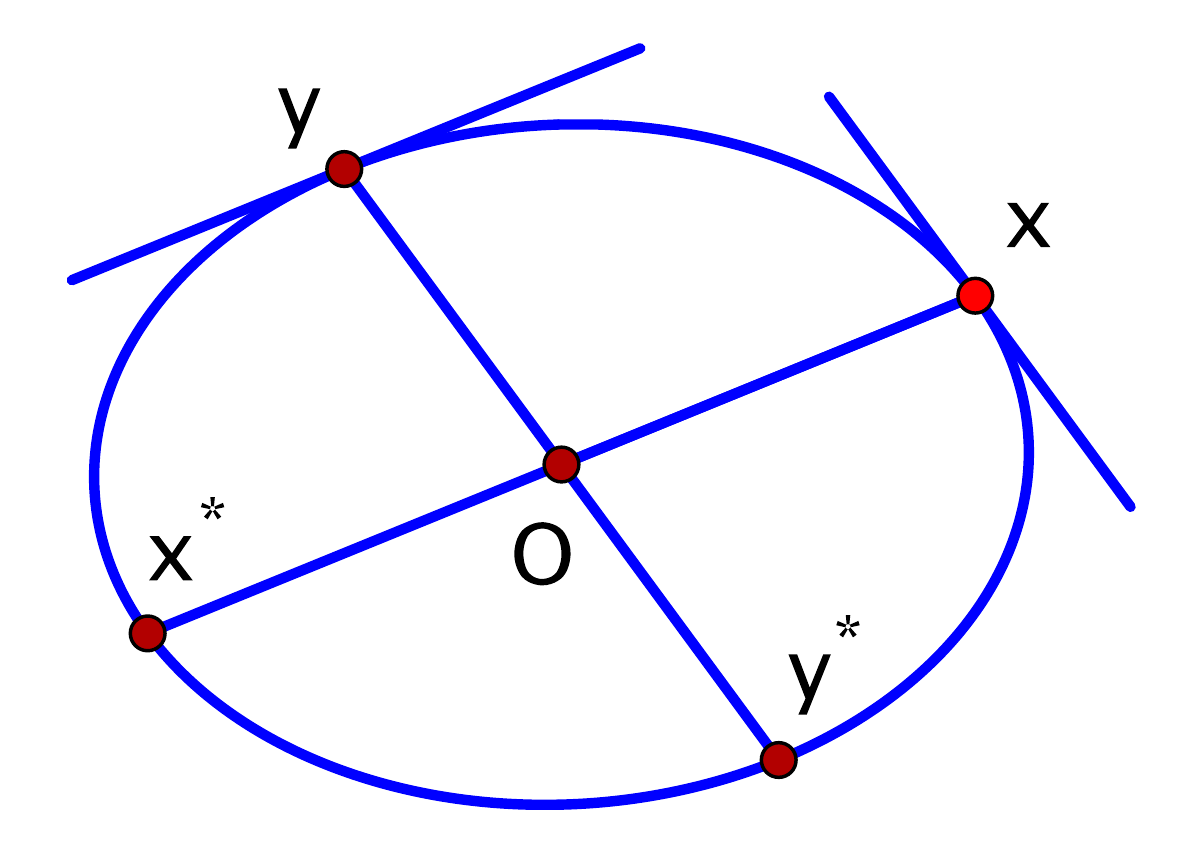}
\caption{The tangent $T_x \g$ is parallel to $Oy$ if and only if $T_y \g$ is parallel to $Ox$, similar to the conjugate diameters of an ellipse.}
\label{Radon1}
\end{figure}

Radon curves are unit circles of a particular kind of Minkowski (normed) planes, called Radon planes, that share many features with Euclidean planes. In particular, in a Radon plane, normality (also called the Birkhoff orthogonality) is a symmetric relation.
Radon curves have been extensively studied since their introduction by Radon in 1916; see \cite{MS} for a modern account. 

The relevance of Radon curves for us is that they have  invariant circles consisting of 4-periodic points: the parallelogram $xyx^*y^*$ is a symplectic billiard orbit for every $x$. Conversely, if a centrally symmetric curve has an invariant circle consisting of 4-periodic points, then it is a Radon curve.

An obvious example of Radon curve is an ellipse. However Radon curves are very flexible, depending on a functional parameter. For instance, here is a construction of a $C^1$-smooth Radon curve, see \cite{MS}. 

Let $a$ and $b$ be two vectors with $[a,b]=1$. Connect $a$ and $b$ by a smooth convex curve $C_1$ that lies in the parallelogram spanned by $a$ and $b$ and is tangent to its sides at the end points. Parameterize $C_1$ by a parameter $t\in[0,T]$ so that $C_1(0)=a,\ C_1(T)=b$, and $[C_1(t),C_1'(t)]=1$ for all $t$. The latter condition means that the rate of change of the sectorial area is constant. Differentiating, we obtain $[C_1(t),C_1''(t)]=0$, hence the acceleration of the curve $C_1$ is proportional to the position vector.

Since $C'_1(0)$ is proportional to $b$, and $[a,b]=1$, we have $C'_1(0)=b$. For the same reason, $C'_1(T)=-a$. 
Now consider the curve $C_2(t)=C'_1(t)$. This curve lies in the second quadrant and connects  $b$ with $-a$. Furthermore, $C_2'(0)=C_1''(0)$ is proportional to $C_1(0)=a$, and $C_2'(T)=C_1''(T)$ is proportional to $C_1(T)=b$. Therefore the union of $C_1, C_2$ and their reflections in the origin is a $C^1$-smooth curve.

Since $[C_1(t),C_2(t)]=1$, one has $[C'_1(t),C_2(t)]+[C_1(t),C'_2(t)]=0$. Thus one summand vanishes if and only if so does the other. This implies  the Radon property. 

For example, one can combine $\ell_p$- and $\ell_q$-norms with $1<p,q < \infty$ and $1/p+1/q=1$, taking $C_1$ and $C_2$ to be the quarters of their respective unit circles, see Figure \ref{Radon}. It is worth mentioning that, as $p\to\infty,\ q\to 1$, the respective Radon curve tends to an affine-regular hexagon.

\begin{remark}
{\rm
The above construction, in general, gives rise to $C^1$-smooth Radon curves. It can obviously adapted to produce $C^k$ or $C^\infty$-smooth examples.
}
\end{remark}

\begin{figure}[hbtp] 
\centering
\includegraphics[width=2in]{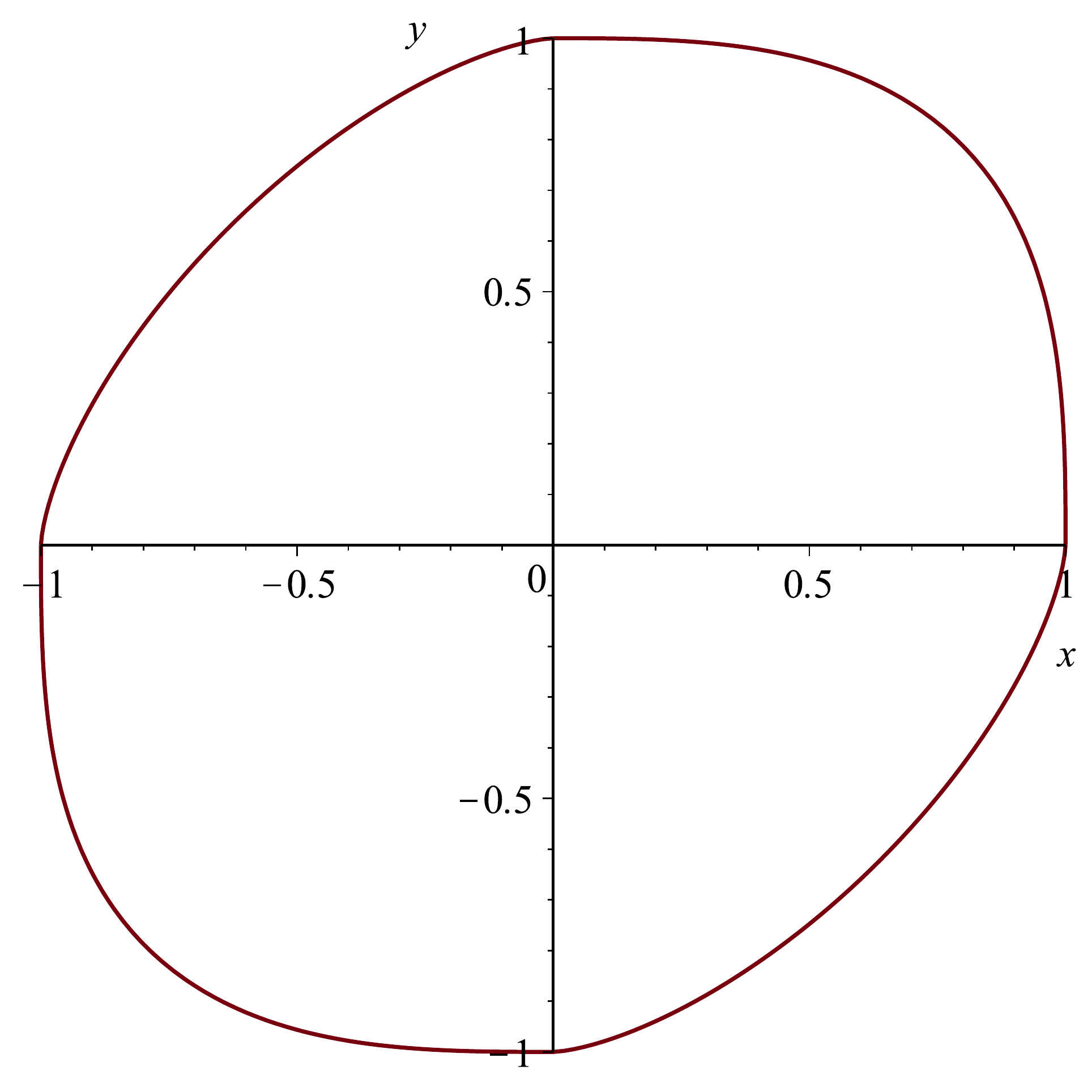}
\caption{A Radon curve combining $\ell_p$- and $\ell_q$-norms with $p=3$ and $q=\frac32$.}
\label{Radon}
\end{figure}

\subsubsection{Scarcity of periodic points} \label{scar}

The second consequence of Theorem \ref{Bdistr} is a version of the $n=3$ and $n=4$ cases of the Ivrii conjecture. This conjecture asserts that the set of periodic trajectories of the usual billiard has zero measure; a seemingly weaker (but, in fact, equivalent) version is that this set has  empty interior.  For period $n=3$, the Ivrii conjecture is proved, by a number of authors and in different ways, in  \cite{BZ,Ry,St,Vo,Wo}; currently, the best known result is for $n=4$, see \cite{GlK}. See also \cite{GT,TZ,Tu} for periods $n=3,4,5,6$ for outer billiards. 

\begin{theorem} \label{Ivrth}
The set of 3-periodic points of the planar symplectic billiard has empty interior. 
If $\g$ is strictly convex, then the set of 4-periodic points also has empty interior.
\end{theorem}

\begin{proof}
First we consider \underline{the case $n=3$.}  The proof is essentially the same as in \cite{BZ} and \cite{GT}.

Let $M$ be the 5-dimensional manifold of triangles of unit area and ${\mathcal D}$ be the above defined distribution on it. We point out that triangles of unit area lie automatically in the space $\mathcal{P}_3$, i.e., $M\subset\mathcal{P}_3$, on which $\mathcal{D}$ is defined. If the set of 3-periodic points of the symplectic billiards map contains an open set, then the respective triangles form a horizontal surface $U \subset M$ (i.e., tangent to $\mathcal{D}$.) Choosing local coordinates in this surface yields a pair of commuting  vector fields on $U$.

Since these fields are horizontal, they can be written in the form 
$f_1 v_1+f_2v_2+f_3v_3$  and $g_1 v_1+g_2v_2+g_3v_3$, where $f_i,g_i$ are functions and $v_i$ are as in (\ref{vectf}).
Without loss of generality, assume that $f_1\neq 0$ and $g_2\neq 0$. Taking linear combinations, we obtain vector fields $w_1=v_1 + fv_3,\ w_2=v_2+gv_3$ with the property that $[w_1,w_2]$ is a linear combination of $w_1$ and $w_2$ with functional coefficients. In particular, $[w_1,w_2]$ is horizontal.

Let $d=[z_2-z_1,z_3-z_1]$ be twice the area of the triangle $z_1z_2z_3$.
One calculates
$$
[w_1,w_2] = [v_1,v_2] -f [v_2,v_3] - g [v_3,v_1] \mod {\mathcal D},
$$
and hence, using (\ref{mat}),
$$
\theta_1([w_1,w_2]) = d(1+g),\ \theta_2([w_1,w_2]) = -d (1+f).
$$
Since $d\ne 0$, it follows that $f=g=-1$, and 
$$
[w_1,w_2] = [v_1,v_2] + [v_2,v_3] + [v_3,v_1] = v_1+v_2+v_3.
$$ 
But this vector field is not a linear combination of $w_1=v_1-v_3$ and $w_2=v_2-v_3$. This contradiction concludes the proof of this case.

Now consider \underline{the case $n=4$.} Let $x_1 x_2 x_3 x_4$ be a 4-periodic orbit in the curve $\g$. Then the tangents $T_{x_1} \g$ and $T_{x_3} \g$ are parallel to the diagonal $x_2 x_4$, and hence $x_3=x_1^*$. Likewise, $x_4=x_2^*$.

Let $\bar x_1 \bar x_2 \bar x_3 \bar x_4$ be another 4-periodic orbit, a perturbation of the first one. Assume that point $\bar x_1$ has moved from point $x_1$ in the positive direction (with respect to the orientation of the curve $\g$). We claim that then the other points $\bar x_i,\ i=2,3,4$, have also moved in the positive direction. We shall refer to this property as the monotonicity condition.

To prove the monotonicity condition, note that since $\g$ is strictly convex, the relation $x\mapsto x^*$ is an orientation preserving involution. Since $\bar x_3=\bar x_{1}^*$, this point has moved in the positive direction. Hence the segment $\bar x_1 \bar x_3$ has turned in the positive sense (compared to $x_1 x_3$). Since $T_{\bar x_2} \g$ is parallel to $\g(\bar x_1) \g(\bar x_3)$, point $\bar x_2$ has moved in the positive direction as well, and so has $\bar x_4=\bar x_{2}^*$.

Now we argue similarly to the $n=3$ case. Let $M$ be the manifold of quadrilaterals of constant area, and let ${\mathcal D}$ be the respective 4-dimensional distribution. Again we claim that 4-periodic points of the billiards map give rise to polygons which automatically lie in $\mathcal{P}_4$. Indeed, by the definition of the symplectic billiard map the characteristic directions of $\g$ at $x_1$ and $x_3$ are parallel to $x_2x_4$ and similarly the characteristic directions at $x_2$ and $x_4$ are parallel to $x_1x_3$. Now, if we assume that $x_1x_3$ and $x_2x_4$ are parallel we find find four points with the same characteristic direction which directly contradicts the strict convexity of $\g$ which allows for exactly two such point.

Therefore, we now assume that there exists a horizontal surface in $M$. Then one has two linearly independent commuting vector fields, $\xi$ and $\eta$, that are lineal combinations of the vector fields $v_1,\ldots,v_4$ with functional coefficients. 

Let $\theta_1,\ldots,\theta_4$ be as above. It follows from (\ref{mat}) that 
$$
\theta_i([v_i,v_{i+1}]) = - \theta_{i+1} ([v_i,v_{i+1}]) = 1,\ i=1,\ldots,4,
$$
because the determinant involved is twice the oriented area of the quadrilateral.

Without loss of generality, assume that the coefficient of $v_1$ in $\xi$ is non-zero. In the following we distinguish two cases. If the coefficient of $v_2$ or of $v_4$ in $\eta$ is non-zero, one can replace $\xi$ and $\eta$ by their linear combinations so that, perhaps after reversing the order of indices, one has
$$
\xi = v_1 + f v_3 + g v_4,\ \eta = v_2 + \bar f v_3 + \bar g v_4,
$$ 
and $[\xi,\eta]$ is a linear combination of $\xi$ and $\eta$ with functional coefficients. We refer to this as case 1.

If the coefficients of $v_2$ and of $v_4$ in $\eta$ vanish then one can replace $\xi$ and $\eta$ by their linear combinations so that $\xi = v_1, \eta = v_3$. This is case 2.

In case 1, 
$$
[\xi,\eta]= [v_1,v_2]-\bar g [v_4,v_1] - f [v_2,v_3] + (f \bar g - \bar f g) [v_3,v_4] \mod {\mathcal D}.
$$
Evaluating $\theta_i,\ i=1,\ldots,4$ on this commutator and equating to zero yields
$$
f = \bar g = -1,\ \bar f g = 0.
$$
Without loss of generality, assume that $g = 0$. Then $\xi = v_1 - v_3$ is a vector field tangent to the disc $U$ consisting of 4-periodic orbits. The flow of this field moves points $x_1$ and $x_3$ in the opposite directions, contradicting the monotonicity condition.

Likewise, in case 2, $\xi = v_1$. The flow of this field moves point $x_1$ in the positive direction, but leaves the other points fixed, again contradicting the monotonicity condition. This completes the proof.
\end{proof}
\begin{remark} \label{parsides}
{\rm As follows from the analysis in section \ref{polygons}, the monotonicity condition does not hold for
4-periodic orbits in a square: in fact, when such an orbit is perturbed, the points $x_1$ and $x_3$ move in the opposite directions, and so do $x_2$ and $x_4$. In particular, this shows that \textit{strict} convexity is necessary.
}
\end{remark}

\subsection{Interpolating Hamiltonians and area spectrum} \label{actspec}

\subsubsection{Area spectrum of symplectic billiard} \label{arsp}

%
Consider the maximal action, that is, the maximal area of a simple $n$-gon, inscribed in the billiard curve $\g$. Let $A_n$ be this area. We are interested in the asymptotics of $A_n$ as $n\to\infty$.

The theory of interpolating Hamiltonians, applied to the symplectic billiard, implies that the symplectic billiard  map equals an integrable symplectic map, the time-one map of a Hamiltonian vector field, composed with a smooth symplectic map that fixes the boundary of the phase space point-wise to all orders, see \cite{Me,MM} and \cite{PS,Si}; see also \cite{Ta1} for an application to  outer billiards.
In particular, this theory provides an asymptotic expansion of $A_n$ in negative even powers of $n$:
$$
A_n \sim a_0 + \frac{a_1}{n^2} + \frac{a_2}{n^4} + \frac{a_3}{n^6} + \ldots
$$ 
In our situation, the coefficient $a_0$ is, of course, the area of the billiard table. The next two coefficients, $a_1$ and $a_2$, were found in \cite{McV,Lud}: 
\begin{equation} \label{a12}
a_1= \frac{L^3}{12},\quad a_2=-\frac{L^4}{240} \int_0^L k(t) dt, 
\end{equation}
where $t$ is the affine parameter on the billiard curve, $L=\int_{\g} dt$ is the total affine length, and $k$ the affine curvature of $\g$. For comparison, for  outer billiards, that is for the circumscribed polygons of the least area, one has $a_1 = L^3/24$. 

In affine geometry one also has an isoperimetric inequality for all strictly convex closed curves, with equality only for ellipses, see \cite{Lut}. It reads
\begin{equation}\label{affine_isop_ineq}
L^3 \le 8\pi^2 A,
\end{equation} 
where $A$ is the area. We point out that it ``goes in the wrong direction'' if compared to the usual isoperimetric inequality. Similarly to \cite{Ta1}, one has the following immediate consequence.

\begin{theorem} \label{hear}
The first two coefficients, $a_0$ and $a_1$, make it possible to recognize an ellipse: one always has the inequality
\begin{equation} \label{affineq}
3 a_1 \leq 2\pi^2 a_0,
\end{equation}
with equality if and only of $\g$ is an ellipse.
\end{theorem}

There is nothing to prove since by (\ref{a12}) the affine isoperimetric inequality (\ref{affine_isop_ineq}) and (\ref{affineq}) are equivalent.

\begin{remark} 
{\rm Of course, we can rephrase Theorem \ref{hear} as ``one can hear the shape of an ellipse". This leads to an interesting open question. Can one interpret the sequence $a_0, a_1, a_2, \ldots$ really as a spectrum? That is, is there a differential operator whose spectrum is this sequence? For usual billiards this is well-known, see for instance \cite{MM}.

Similarly it would be very interesting to determine the higher terms $a_3, a_4, \ldots$, even in the case of ellipses, directly. In fact, for ellipses there is a little miracle that $a_0$ and $a_1$ determine all other terms. This is the affine isoperimetric inequality.
}
\end{remark}

\subsubsection{Insecurity of symplectic billiards}   \label{security}
 A classical billiard is called secure (or has the finite blocking property) if, for every two points $A$ and $B$ in the billiard table, there exists a finite set of points $S$, such that every billiard trajectory from $A$ to $B$ visits $S$ (so the set $S$ blocks $A$ from $B$). For example, billiard in a square is secure.

It is proved in \cite{Ta3} that planar billiards with smooth boundary are not secure. In a nutshell, the argument is as follows. 

Let $A$ and $B$ be on the convex part of the boundary curve $\g$, and consider the shortest $n$-link billiard trajectory from $A$ to $B$. For sufficiently large $n$, no points that are not on the boundary can block this trajectory. 

Using the theory of interpolating Hamiltonians, one shows that, modulo errors of order $1/n^2$, the reflection points are regularly distributed on the arc $AB$ with respect to the measure $\kappa ^{2/3} dx$, where $x$ is an arc length parameter, and $\kappa$ is the curvature of $\g$. 

One can introduce a coordinate $t$ so that $dt = \kappa ^{2/3} dx$ and normalize so that $t \in [0,1]$ on the arc $AB$. If the reflection points were regularly distributed, and there were $n$ reflections, then the reflection points would have coordinates $m/n,\ 1\leq m < n$. Then it is clear that for every finite set $S \subset [0,1]$, there exists $n$ such that $S$ contains no fractions with denominator $n$: one can take $n$ to be a prime number greater than all the denominators of the rational numbers contained in $S$.

The actual argument uses some number theory (Proposition 2 or the more general Theorem 3.2 of \cite{Ta3}) to deal with the errors of order $1/n^2$ in the distribution of the reflection points. To recap, no finite set of points on the arc $AB$ can block all billiard trajectories from $A$ to $B$.

One can use the same approach to prove an analogous result for symplectic billiards that we formulate below without proof. The key observation is contained in \cite{Lud}: with respect to the affine parameter $t$, the 
vertices of inscribed $n$-gons of maximal area, that is, the reflection points of the $n$-link symplectic billiard trajectory, are equidistributed  modulo errors of order $1/n^2$. 

\begin{theorem} \label{insec}
The symplectic billiard inside a smooth strictly convex curve is insecure. More precisely, for every pair of distinct points $x,y$ on the boundary curve $\g$ and every finite set $S\subset \g$, there exists a symplectic billiard trajectory from $x$ to $y$ that avoids $S$.
\end{theorem}

\section{Polygons} \label{polygons}

In this section we collect a few simple results on polygonal symplectic billiards. In our opinion, this subject deserves a thorough study.

We start with a remark concerning the definition: the symplectic billiard map is not defined for a chord $xy$ if the sides of the polygon that contain these points are parallel. The map is also not defined if point $x$ is a vertex of the polygon. Note that if the end points of a segment of a billiard trajectory are not on parallel sides of a polygon, then the same is true for the next segment of the trajectory. 

Let $P$ be a convex polygon.
The phase space $\P$ of the symplectic billiard is the torus $P\times P$, and it is naturally decomposed into rectangles, the products of pairs of sides of $P$ (if $P$ has pairs of parallel sides then the map is not defined on the respective rectangles). The symplectic billiard map $\Phi$ is a piecewise affine map of this phase space.

\subsection{Regular polygons} \label{regul}

The case of regular polygons is very interesting in both `classical' cases, the inner and the outer billiards. In the  case of inner billiards, the Veech Dichotomy holds: for every direction in a regular polygon, the billiard flow is either periodic or uniquely ergodic, see, e.g., \cite{HS,MT}, just like in the well known case of a square. In the case of outer billiards, (affine) regular polygons have an intricate fractal orbit structure (except for $n=3, 4, 6$ when all orbits are periodic) which is analyzed only for $n=5,8$, see \cite{Ta,Sc}.

Let $P$ be a regular $n$-gon whose sides are cyclically labeled $0,1,\ldots, n-1$. Assume that the initial segment of a billiard trajectory connects side $0$ with side $k$; here $1\leq k \leq [(n-1)/2]$. We shall call $k$ the rotation number of the trajectory. See Figures \ref{3and4} 
and \ref{5and6}. 

\begin{figure}[hbtp] 
\centering
\includegraphics[width=2.2in]{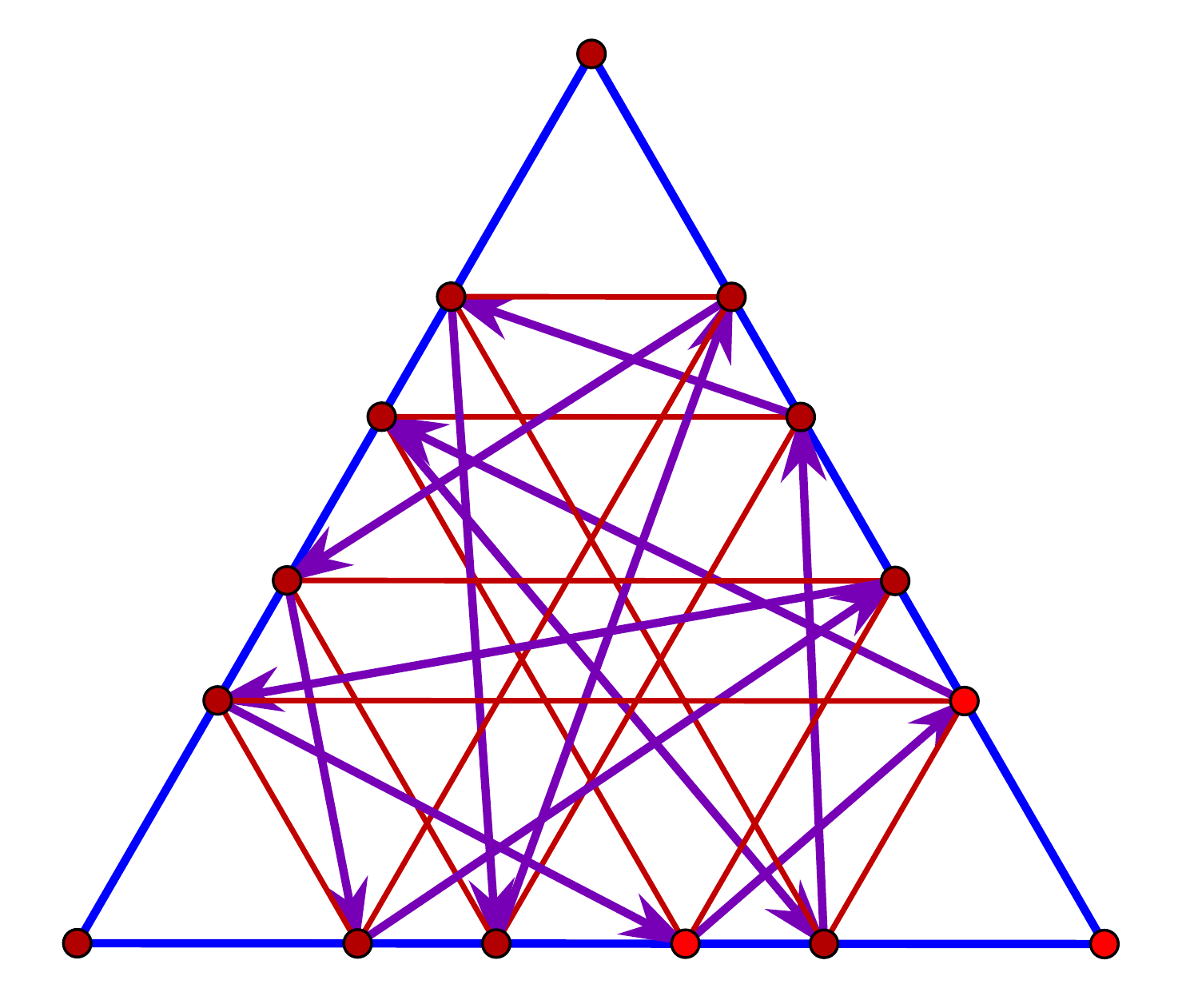}
\qquad
\includegraphics[width=2in]{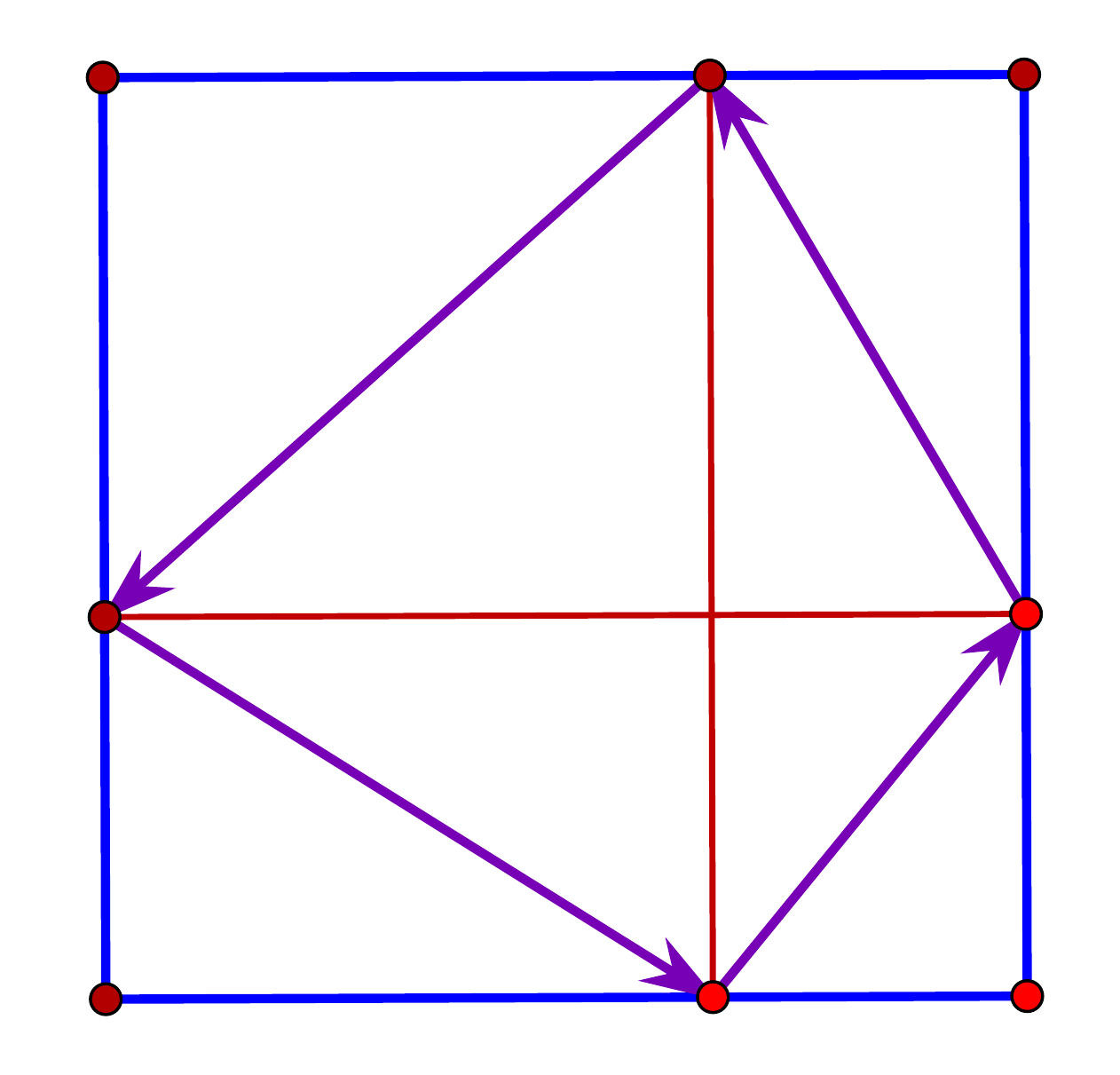}
\caption{A $12$-periodic orbit in a triangle and a $4$-periodic orbit in a square; in both cases, $k=1$.}
\label{3and4}
\end{figure}

\begin{figure}[hbtp] 
\centering
\includegraphics[width=2.1in]{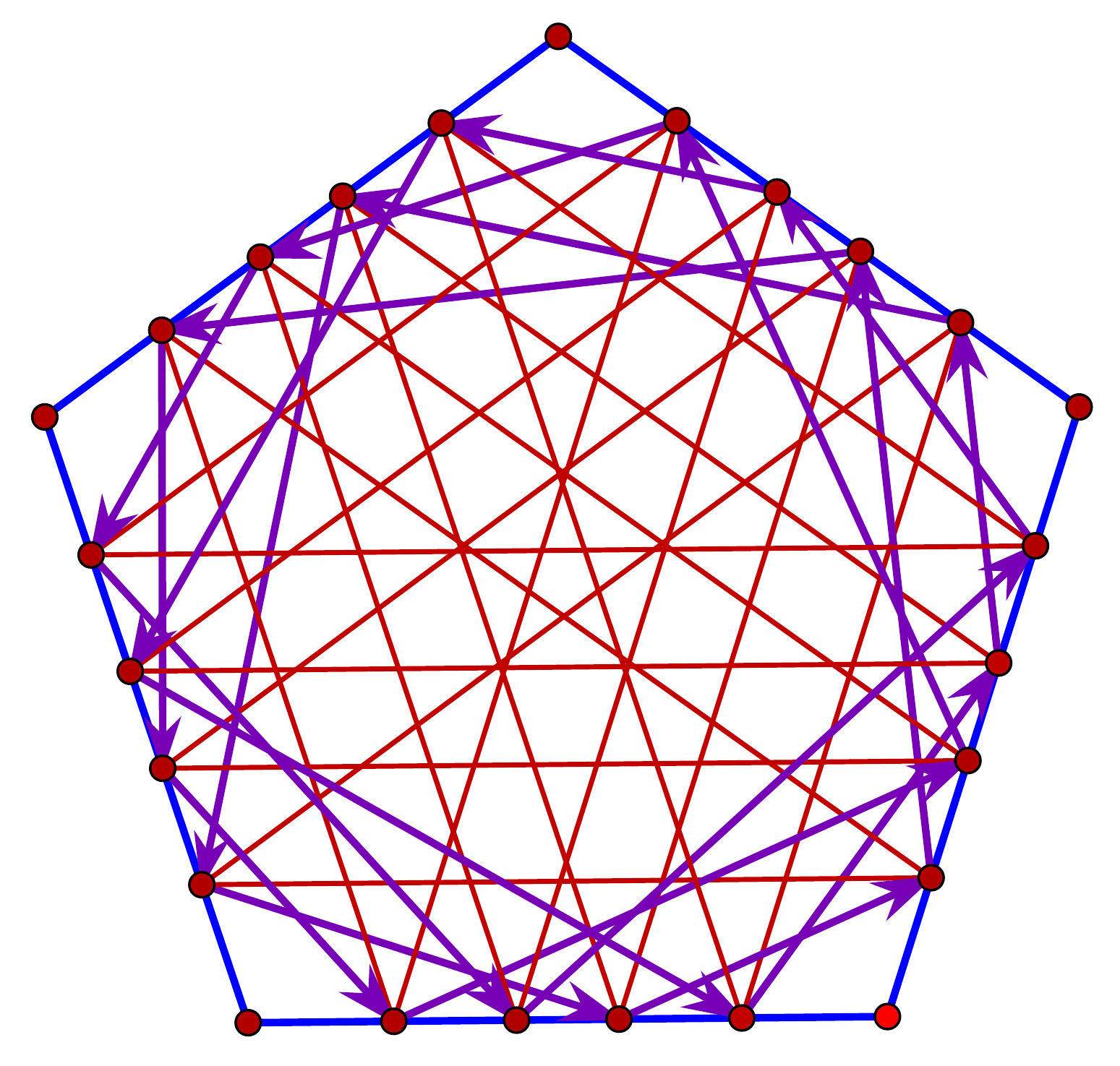}
\qquad
\includegraphics[width=2.2in]{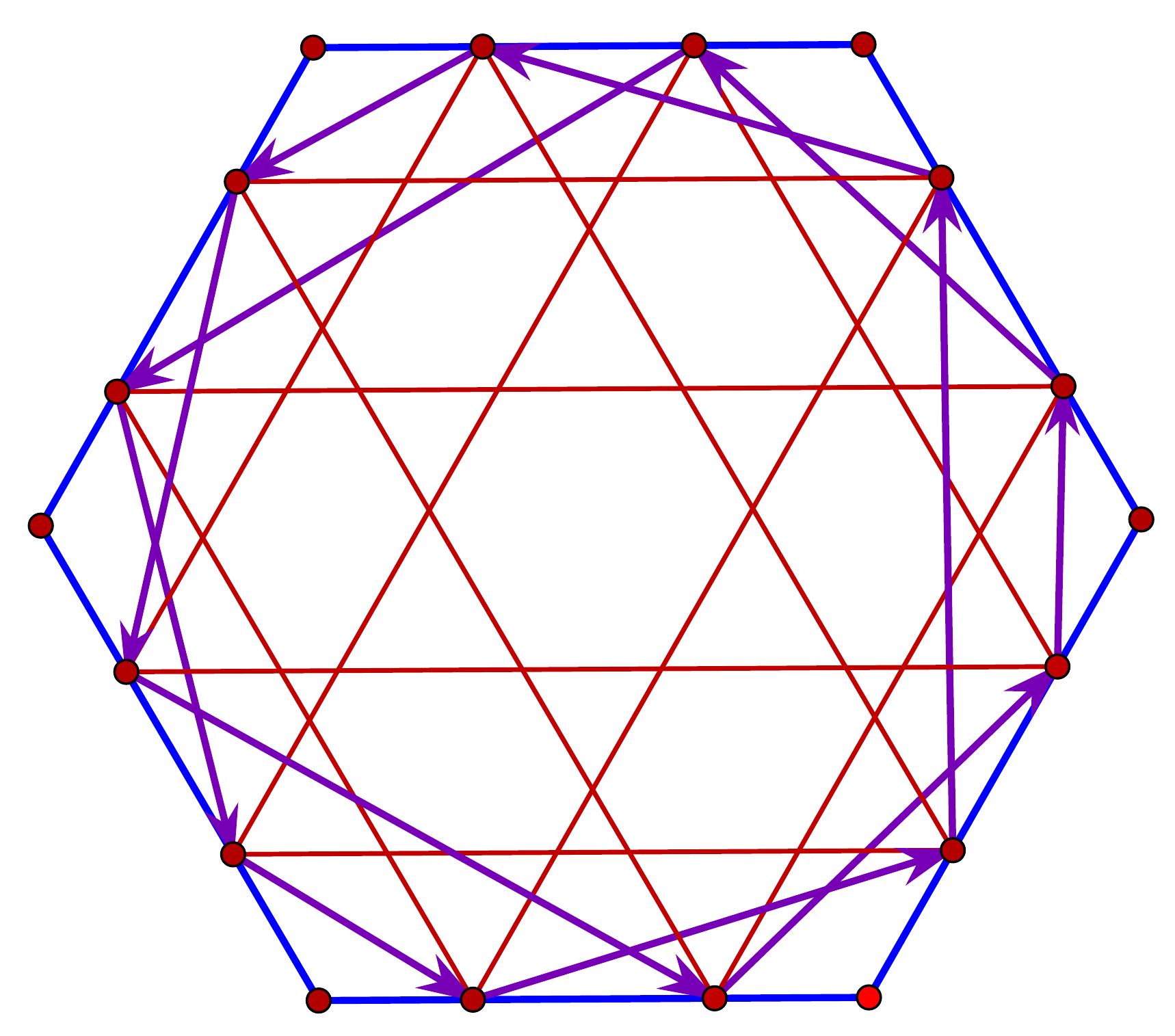}
\caption{A $20$-periodic orbit in a regular pentagon and a $12$-periodic orbit in a regular hexagon; in both cases, $k=1$.}
\label{5and6}
\end{figure}

\begin{theorem} \label{regpol}$ $
\begin{enumerate}  \itemsep=1.5ex
\item The rotation number of an orbit is well defined: each link of the orbit connects side $i$ with side $i+k$.
\item Let 
$$
g(n,k) = \frac{n}{{\rm gcd}(n,2k)}.\\[1ex]
$$
\begin{enumerate}  \itemsep=1ex
\item If $g(n,k)$ is even, then the respective billiard orbits are $2g(n,k)$-periodic.
\item If $g(n,k)$ is odd, the orbits are $4g(n,k)$-periodic.
\end{enumerate}
\end{enumerate}
\end{theorem}

\begin{proof}Let the orbit be $x_0,x_1,\ldots$ with $x_0$  on the side labeled $0$, and $x_1$ on the side labeled $k$. Due to the dihedral symmetry of the polygon, the projection along the $k$th side takes  the $0$th side to the $(2k)$th side. Hence $x_2$ lies on the $(2k)$th side, and so on. It follows that the rotation number of the orbit is well defined. 

The segments connecting $x_0$ to $x_2$ to $x_4$, etc.~are parallel to the sides of the polygon, and likewise for the segments connecting $x_1$ to $x_3$ to $x_5$, etc. One obtains two polygonal lines, even and odd. Due to the symmetry of a regular polygon $P$, both are usual billiard trajectories in $P$, see Figure \ref{pic_reg_gons}. We refer to them as the even and odd (usual) billiard trajectories. 

\begin{figure}[hbtp] 
\centering
\includegraphics[width=5.5in]{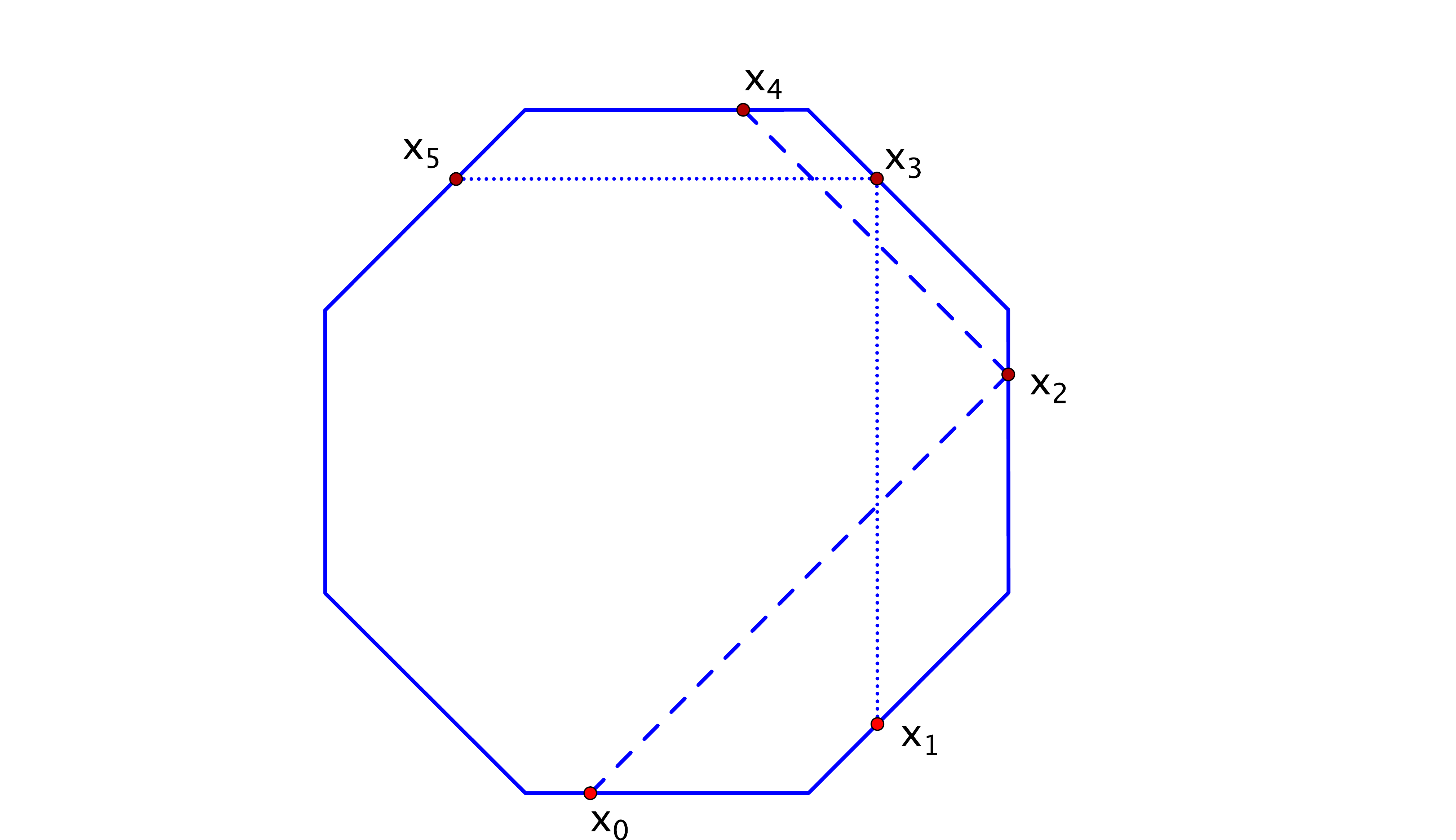}
\caption{Pieces of the even (dashed) and odd (dotted) billiard trajectories. The symmetry of the regular octagon enforces equal angles at $x_3$ and $x_4$.}
\label{pic_reg_gons}
\end{figure}

Consider a particular case of such a billiard trajectory, the one whose initial segment connects the midpoint of side $0$ to the midpoint of side $2k$. This billiard trajectory is periodic with period $g(n,k)$. A parallel trajectory is also periodic, with period $g(n,k)$ if $g(n,k)$ is even, and period $2g(n,k)$ if $g(n,k)$ is odd. 

Next we observe that the even billiard trajectory depends only on the choice of the point $x_0$ (and the fixed number $k$), and the odd one only on the choice of the point $x_1$. This implies that, generically, these two billiard trajectories are different. Therefore the total number of vertices of the orbit $x_0,x_1,\ldots$ is $2g(n,k)$ if $g(n,k)$ is even, and  $4g(n,k)$ if $g(n,k)$ is odd. 
\end{proof}
\begin{corollary}
{\rm Since all triangles are affine equivalent, all orbits in triangles are periodic, generically, of period 12.
}
\end{corollary}

\subsection{Trapezoids} \label{trapsec}

Up to affine transformations, trapezoids form a 1-parameter family; it is convenient to normalize them so that they are isosceles. We assume that the lower horizontal side $AB$ is greater than the upper one, $CD$. Define the modulus of a trapezoid $ABCD$ to be
$$
 \left[ \frac{|AB|}{|AB|-|CD|} \right].
$$
The modulus is a positive integer. Let us call a trapezoid generic if $|AB|/(|AB|-|CD|)$ is not an integer. For example, one may consider a triangle (with $|CD|=0$) as a trapezoid of modulus one, but it is not generic.

\begin{theorem} \label{trapezoid}
All billiard orbits in a trapezoid are periodic. If the modulus of a generic trapezoid is $n$, then there are three periods: $16n-4, 16n+4$, and $16n+12$.
\end{theorem}

\begin{proof}

Let $x_0, x_1, x_2, \ldots$ be a trajectory.  Similarly to the proof of Theorem \ref{regpol}, we consider the even and odd subsequences $x_0, x_2, x_4, \ldots$ and $x_1, x_3, x_5, \ldots$ separately.  

Let us define a transformation $F$ of the (boundary of the) trapezoid as follows.
Let $X$ be an interior point of a side of the trapezoid. Through $X$ there pass two lines parallel to one of the sides of the trapezoid and intersecting its interior. Choose one of them and move the point $X$ in this direction until it lands on the trapezoid  at point $Y$. Through $Y$ there again pass two lines parallel to one of the sides of the trapezoid. One of them takes $Y$ back to $X$; choose the other one and move the point $Y$ to the new point $Z$, etc. We have described the map $F: X \to Y \to Z \to \ldots$.

\begin{figure}[hbtp] 
\centering
\includegraphics[width=2.5in]{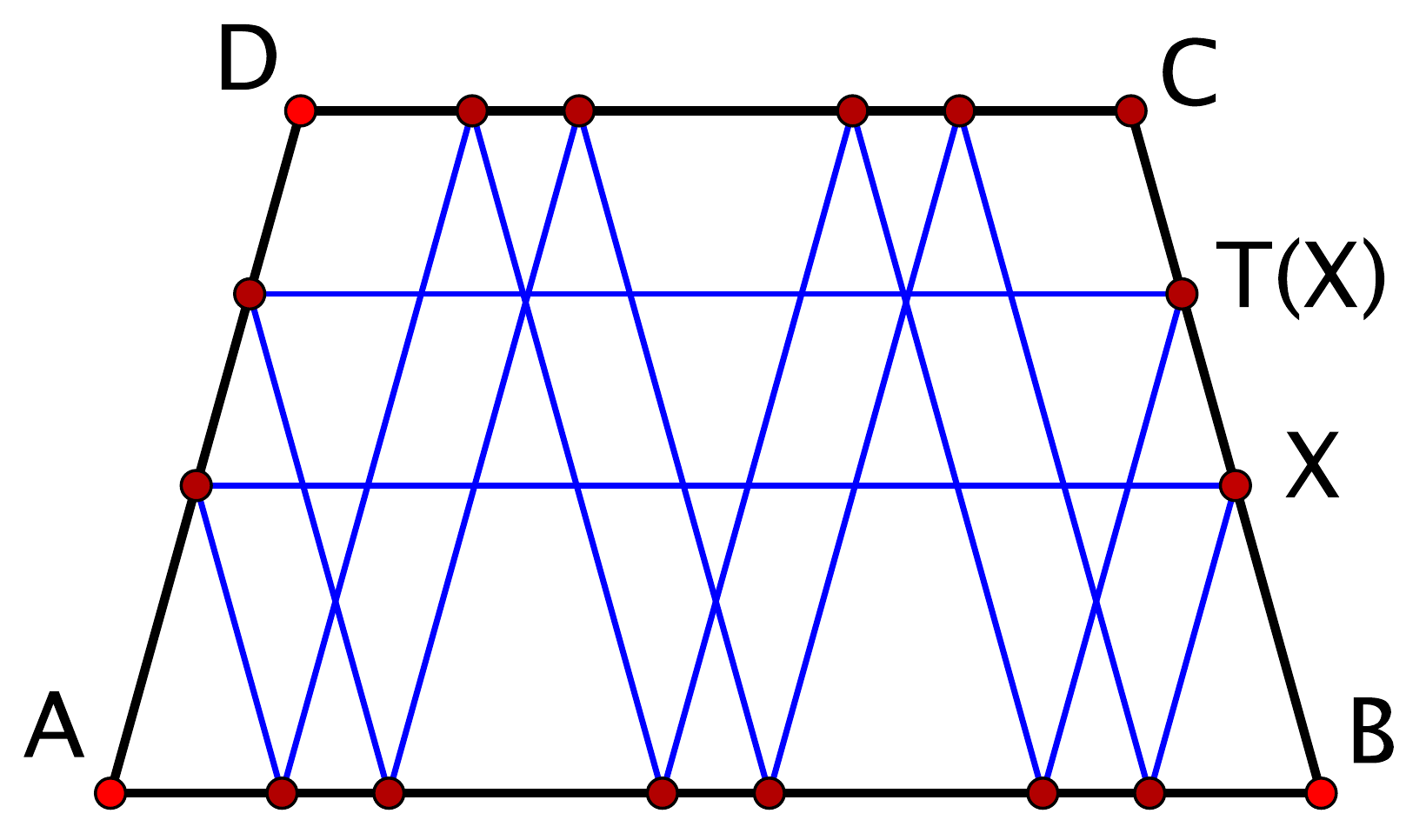}\quad
\includegraphics[width=2.5in]{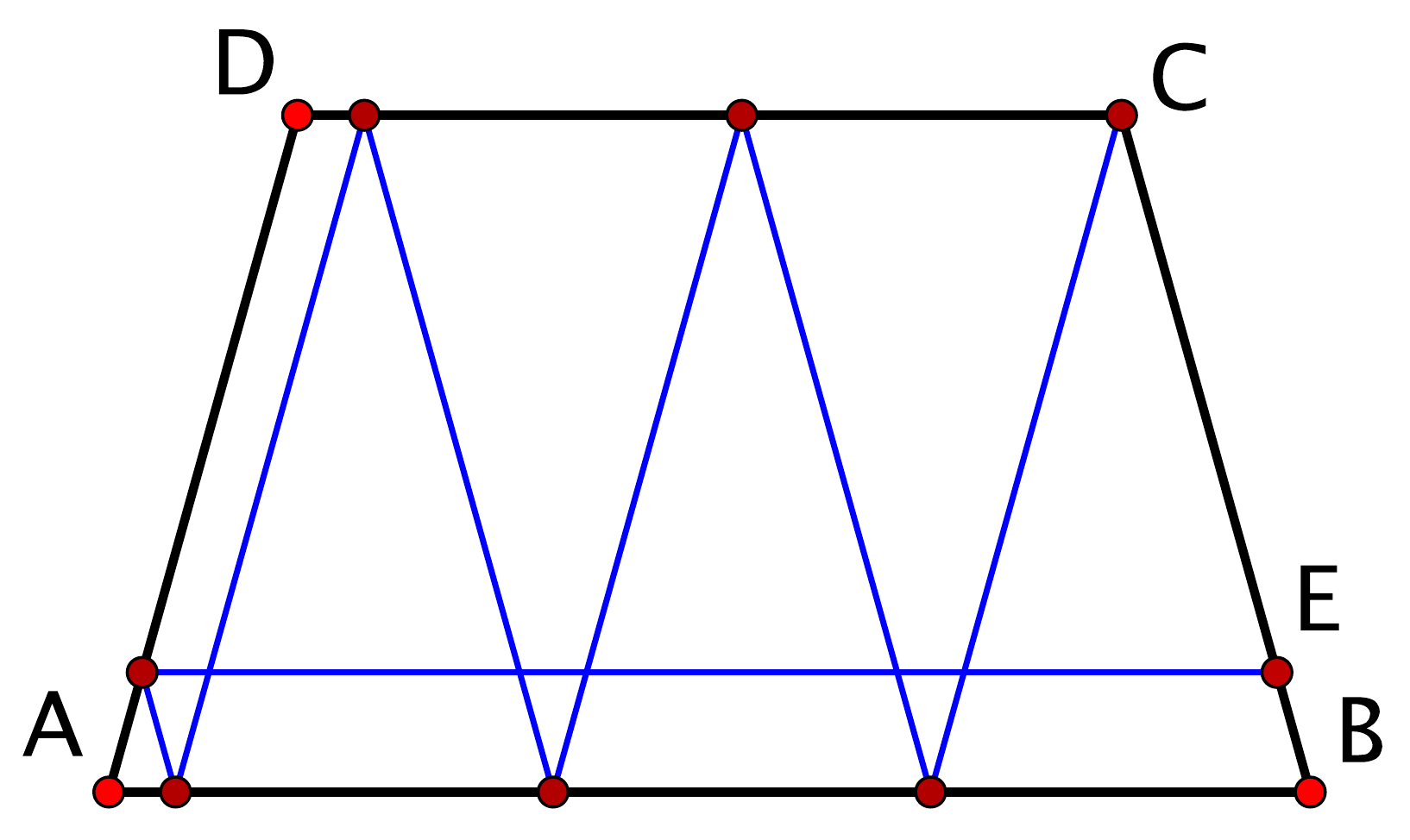}
\caption{Left, the return map $T$. Right, $E$ is a break point: $T(E)=C$.}
\label{trap}
\end{figure}

Since every trajectory necessarily visits all sides we take the starting point on the side $BC$ and move horizontally left. 
After two moves, the point lands on the  horizontal side $AB$ and makes an even number of `bouncing' moves between the horizontal sides, after which one more move in the direction $AD$ takes the point to side $BC$. Thus we have a return map $T$ to the side $BC$, see Figure \ref{trap}.

There is a break point $E$ on the side $BC$ such that $T(E)=C$, see Figure \ref{trap} on the right. For $X \in CE$, the    number of the bouncing moves between the horizontal sides equals $2(n-1)$, and for $X \in BE$, this number is $2n$. 

Note that the map $T$ consists of an odd number of moves, hence it reverses orientation.   In addition, $T$ is a local isometry: the bouncing between the horizontal sides does not distort the length, and the two moves, from a side  $AD$ to $AB$ and from $AD$ to $BC$, distort the length by reciprocal factors. It follows that $T$ is a reflection in a point, namely, the mid-point of the segment $BE$ or $EC$. 
Thus we find the period of point $X$ under the map $T$: it equals $4n+2$ if $X \in CE$, and $4n+6$ if $X \in BE$. 

In what follows we shall call the polygonal lines connecting the consecutive images of points under the map $F$ the guides. A guide that is an $4n+2$-gon is called short, and a guide that is an $4n+6$-gon is called long.

Now we can claim that every orbit of the symplectic billiard map $\Phi$ is periodic. Indeed, let 
$x_0 x_1$ be the initial segment of an orbit. As we have proved, the point $x_0$ determines a finite set of possible positions for even-numbered points $x_i$, and likewise, the point $x_1$ determines a finite set of possible odd-numbered points $x_j$. Therefore there are only finitely many possibilities for each segment $x_i x_{i+1}$, and 
hence the orbit is periodic.

Now we calculate the periods. Let $x_0,x_1,\ldots$ be an orbit. Consider the guides of points $x_0$ and $x_1$. Each one of them can be either short or long. Respectively, one needs to consider three cases: short-short, short-long, and long-long. The combinatorics of the orbits from each case are the same, but  differ from case to case. This is why one has three different periods.

Let us describe the short-short case in  detail. After that, we indicate how the other two cases differ. 

\begin{figure}[hbtp] 
\centering
\includegraphics[width=3.2in]{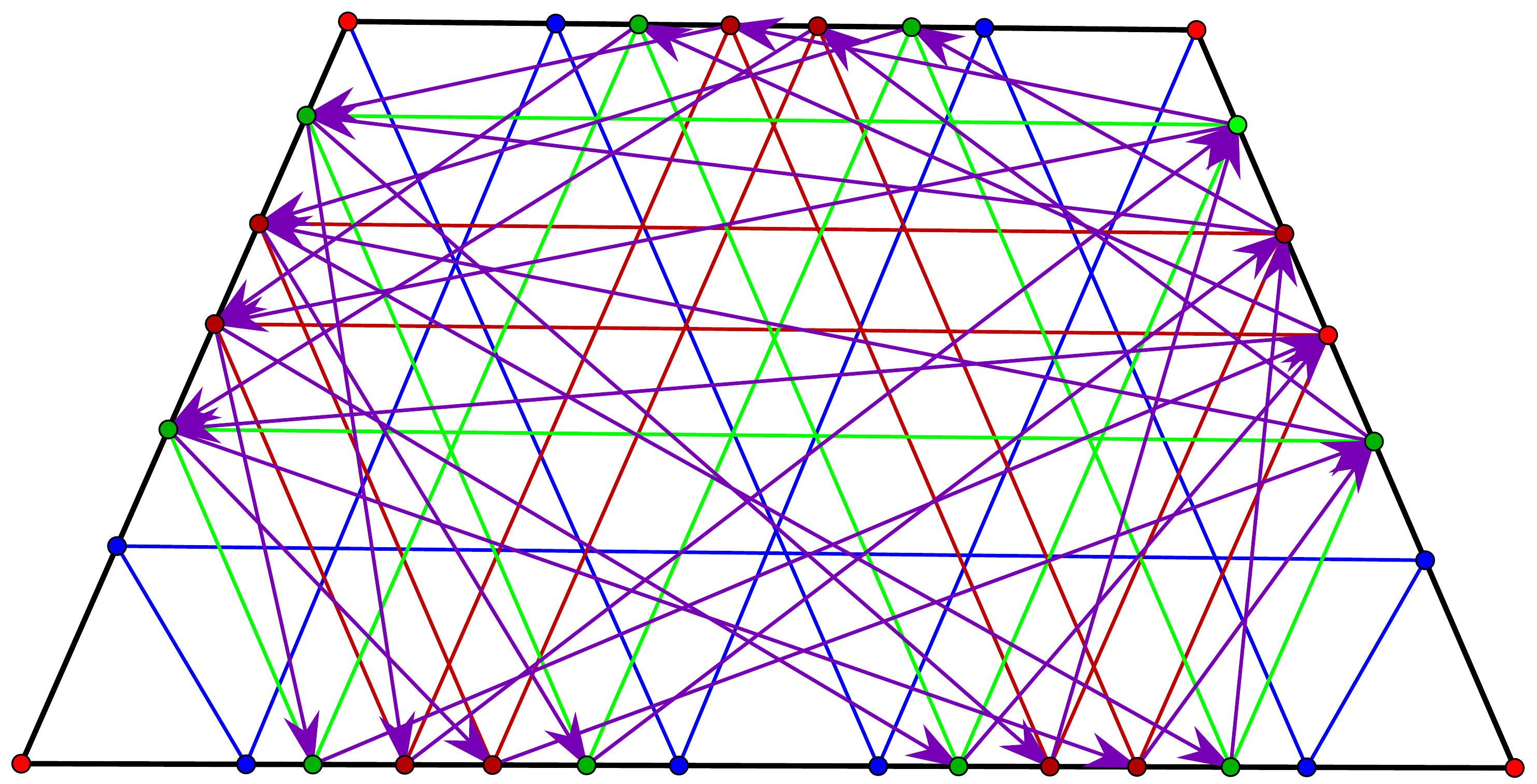}
\caption{A symplectic billiard orbit in a trapezoid of modulus 2, short-short case.}
\label{mess}
\end{figure}

Consider Figure \ref{mess} that depicts a trapezoid of modulus 2: one sees a 28-periodic orbit and the two short guides, each a decagon. Even for small $n=2$, this picture is already a mess. It is easier to analyze a particular case when the two guides coincide and contain the fixed point of the map $T$, see Figure \ref{trimm}. The effect of considering this special case is making the period four times as small (in other words, when we move the guides apart, the number of arrows quadruples).   

\begin{figure}[hbtp] 
\centering
\includegraphics[width=3.2in]{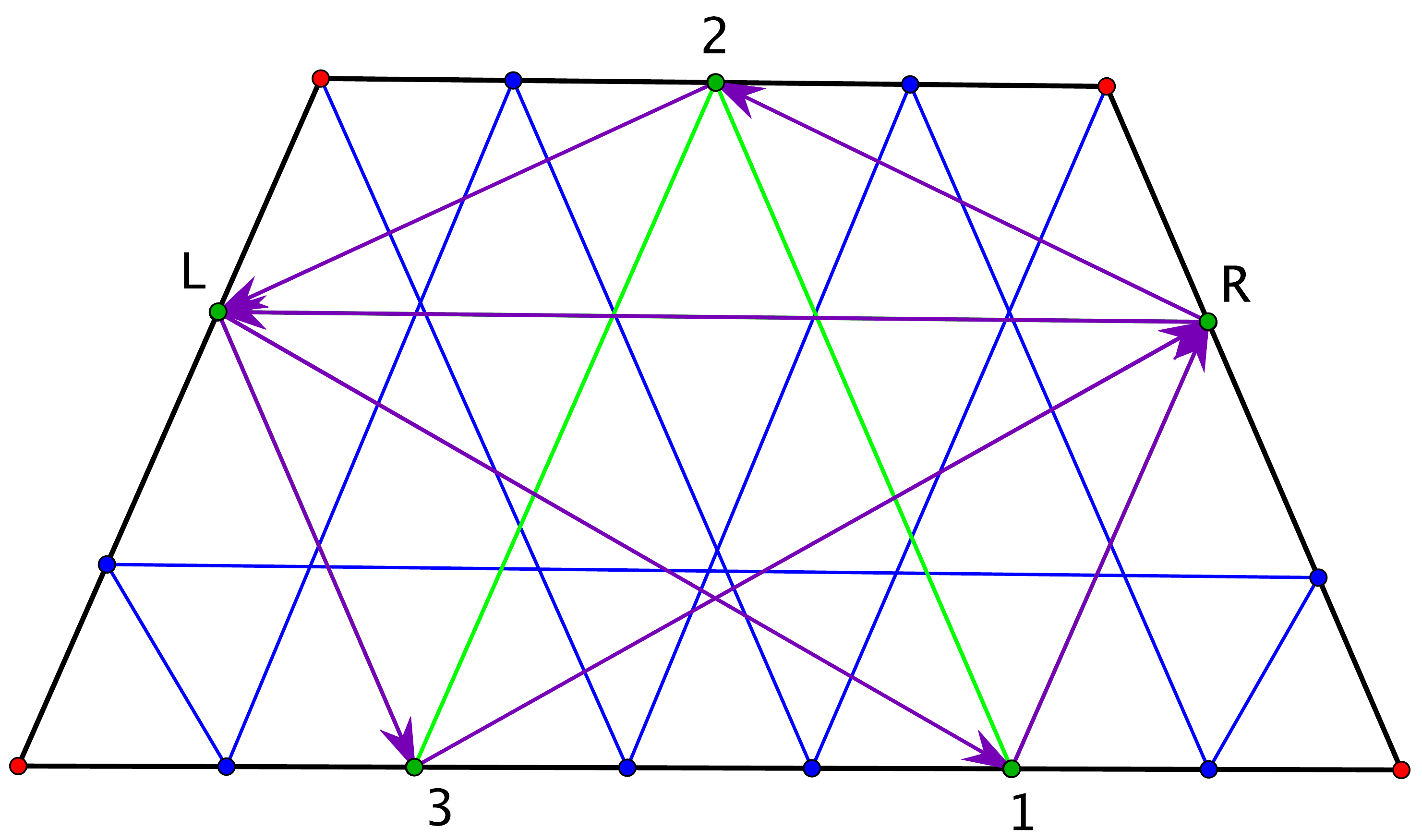}
\caption{A trimmed version: the fixed point of the map $T$, labeled $R$,  is both an even and odd vertex of the orbit.}
\label{trimm}
\end{figure}

Let us analyze the combinatorics. 
The guide has $2n+1$ vertices: two on the lateral sides, $n$ on the side $AB$, and $n-1$ on the side $CD$ (we use the notation in Figure \ref{trap}). Let us label these points along the guide as shown in Figure \ref{trimm} (the odd-numbered points are on side $AB$ and the even-numbered ones on side $CD$). Then the symplectic billiard orbit is 
$4n-1$-periodic and has the following code:
$$
(1\ R\ 2\ L\ 3\ R\ 4\ L\ \ldots\ (2n-1)\ R\ L).
$$
Quadrupling this number yields the short-short period $16n-4$. 

\begin{figure}[hbtp] 
\centering
\includegraphics[width=3.2in]{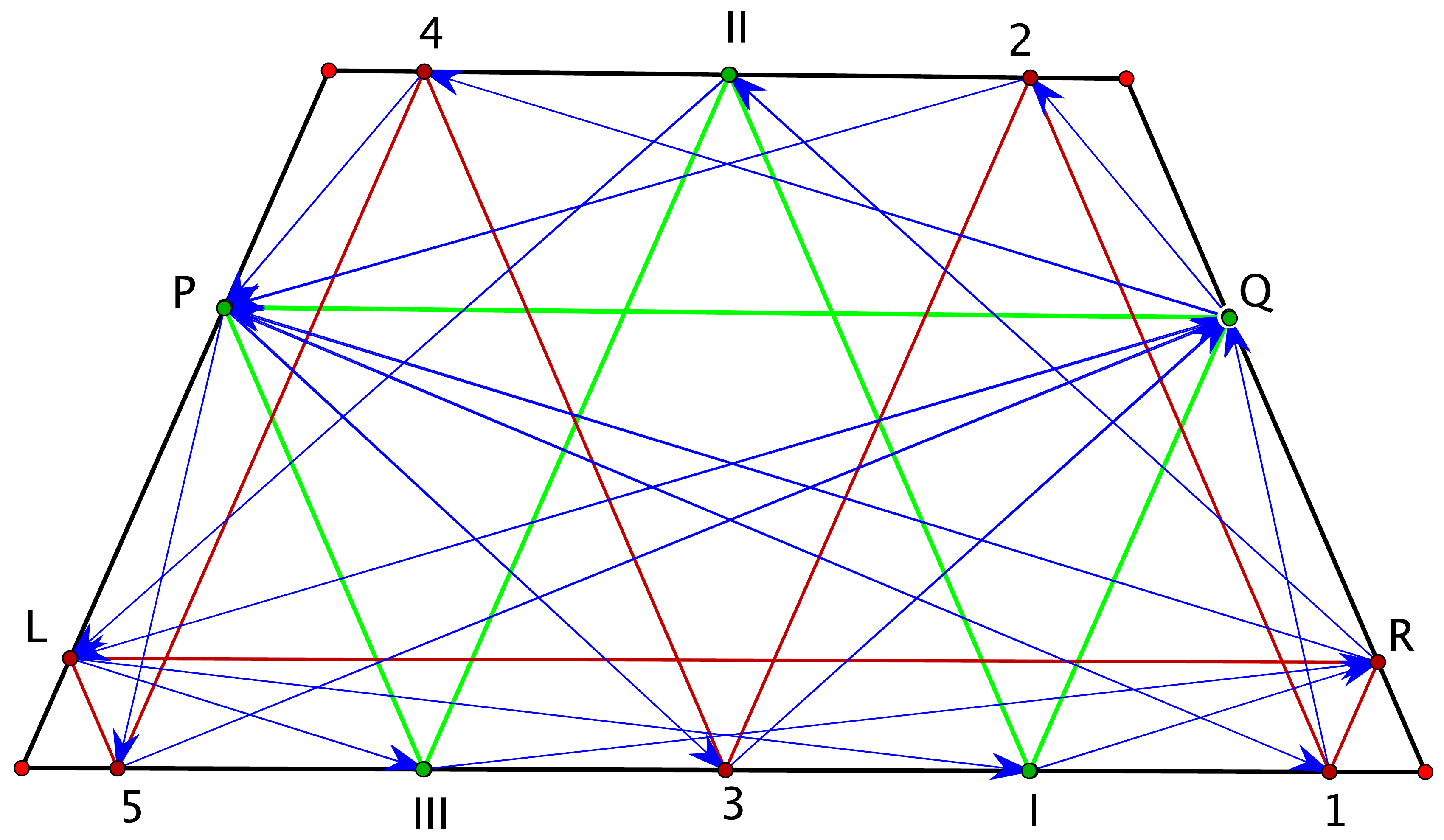}
\caption{A trapezoid of modulus 2, the short-long case.}
\label{shortlong}
\end{figure}

The long-long case is similar to the short-short one. 

The short-long case is illustrated in Figure \ref{shortlong}. Once again, we place the two guides in a special position and label their vertices as shown in the figure. This time the effect of the special position is in halving the number of arrows. The symbolic orbit in the case of $n=2$ is as follows:
$$
(L \ \text{I}\ R\ \text{I\!\,I}\ L\ \text{I\!\,I\!\,I}\ R\ P\ 1\ Q\ 2\ P\ 3\ Q\ 4\ P\ 5\ Q)
$$
and, in general, the period consists of $1+ 2(2n-1) + 2(2n+1) +1 = 8n+2$ symbols. Doubling this numbers gives the short-long period $16n+4$.
\end{proof}
\begin{remark}
{\rm It is an interesting problem to describe those polygons in which all symplectic billiard orbits are periodic. Another problem is to determine whether every convex polygon possesses a periodic symplectic billiard orbit.
}
\end{remark}

\subsection{Stable and unstable periodic trajectories} \label{stab}

It is known, and easy to prove, that an $n$-periodic point of a polygonal outer billiard transformation has a neighborhood consisting of periodic points of period $n$, if $n$ is even, and period $2n$, if $n$ is odd, see for instance \cite{DT}. A similar property holds for odd-periodic orbits of the polygonal symplectic billiard. For the statement of the following Propositions we need a little preparation.

Let $(x_0,\ldots,x_{n-1})$ be an $n$-periodic orbit. Let $\ell_i$ be the line through point $x_i$ parallel to $x_{i-1} x_{i+1}$. One cannot reconstruct the billiard polygon $P$ from the periodic orbit, but one knows that the side of $P$ through point $x_i$ lies on the line $\ell_i$.

Let $(\bar x_0 \bar x_1), \bar x_0 \in \ell_0, \bar x_1 \in \ell_1$, be the initial segment of a nearby billiard trajectory. As in the proof of Theorem \ref{regpol}, the evolutions of points $\bar x_0$ and $\bar x_1$ decouple: $\bar x_0$ is projected to line $\ell_2$ along $\ell_1$, then to $\ell_4$ along $\ell_3$, etc., and likewise for $\bar x_1$. Denote the projection of $\ell_{i-1}$ to $\ell_{i+1}$ along $\ell_i$ by $\pi_i$.

Let $\alpha_{i+\frac12}$ be the angle between $\ell_i$ and $\ell_{i+1}$. Let $dl_i$  be the length element on the line $\ell_i$. Then the distortion of length under the affine map $\pi_i$ is given by the formula:
$$
\frac{dl_{i+1}}{dl_{i-1}}=\frac{\sin \alpha_{i-\frac12}}{\sin \alpha_{i+\frac12}}.
$$
If $n$ is odd, then the first time that point $\bar x_0$ returns to line $\ell_0$ is after $n$ projections, that is, under the affine map 
$\pi_{n-1} \ldots \pi_3 \pi_1$. This map is orientation reversing and it is an isometry because
$$
\prod_{i=1}^n \frac{\sin \alpha_{i-\frac12}}{\sin \alpha_{i+\frac12}} =1.
$$
Hence the second iteration of this map is the identity. 

A similar argument applies to $\bar x_1$. Thus the orbits of both points, $\bar x_0$ and $\bar x_1$ consist of $2n$ points, and altogether, the billiard orbit is $4n$-periodic. Moreover, since the map is the identity, an entire neighborhood of this orbit is periodic with the same period. This proves the following.

\begin{proposition} \label{nbhd1}
An $n$-periodic point in phase space with $n$ being odd has an open neighborhood in phase spave which consists of $4n$-periodic orbits.
\end{proposition}

If $n$ is even, then the point $\bar x_0$ returns to the line $\ell_0$ after $\frac{n}{2}$ projections. The derivative of this affine map is
$$
\prod_{i\ {\rm odd\ mod}\ n} \frac{\sin \alpha_{i-\frac12}}{\sin \alpha_{i+\frac12}}.
$$
If this product is not equal to 1 then the fixed point $x_0$ is isolated. A similar argument applies to point $\bar x_1$, with the derivative given by the formula
$$
\prod_{i\ {\rm even\ mod}\ n} \frac{\sin \alpha_{i-\frac12}}{\sin \alpha_{i+\frac12}},
$$
i.e., if this product is not equal to 1 then $x_1$ is isolated.

In conclusion, if the derivatives of first return maps to $\ell_0$, resp.~to $\ell_1$, at the fixed point are both different from 1 then the fixed points are hyperbolic. Hyperbolic fixed points do not vanish under small perturbations of the map, and hence of the polygon. This proves the following.

\begin{proposition} \label{nbhd2}
If $n\ge 6$ is even then, for a generic polygon, an $n$-periodic point is isolated and is stable in the sense that it does not disappear under a small perturbation of the polygon.
\end{proposition}

\section{Symplectic billiard in symplectic space} \label{higher_dimensions}

\subsection{Definition and continuous limit} \label{defhd}

\subsubsection{A precise look at the definition} \label{hdprecise}

We now extend the definition of the symplectic billiard map to linear symplectic space $\R^{2n}$. 

Consider a smooth closed hypersurface $M\subset \R^{2n}$ bounding a strictly convex domain. On $M\times M$ we consider the  function 
\begin{equation*}
S:M\times M \to \R,\ 
(x_1,x_2) \mapsto S(x_1,x_2):=\omega(x_1,x_2).
\end{equation*}
In analogy to Section \ref{refllaw}, we define the (open, positive part of) phase space as
\begin{equation*}
\P:=\{(x_1,x_2)\in M\times M\mid \omega(\nu_{x_1},\nu_{x_2})>0\}.
\end{equation*}
Here $\nu$ is again the outer normal. Since $M$ is strictly convex, the Gauss map $M\ni x\mapsto \nu_x\in S^{2n-1}$ is a bijection. Denote by $R(x)=\ker\omega|_{T_xM\times T_xM}$ the characteristic direction of $M$ at the point $x$. 

\begin{lemma}\label{omega_0_equiv_charc_direction_in_tangent_space}
The relation $R(x_2)\subset T_{x_1}M$ is equivalent to $\omega(\nu_{x_1},\nu_{x_2})=0$. In particular, it is symmetric in $x_1$ and $x_2$:
\begin{equation*}
R(x_2)\subset T_{x_1}M \quad \Longleftrightarrow \quad R(x_1)\subset T_{x_2}M.
\end{equation*}
\end{lemma}

\begin{proof}Since 
$R(x_2)=\R\cdot(J\nu_{x_2}),
$
where $J$ in the standard complex structure on $\R^{2n}\cong\C^n$, we observe
\begin{equation*}
R(x_2)\subset T_{x_1}M \quad \Longleftrightarrow \quad \langle J\nu_{x_2},\nu_{x_1} \rangle=0 \quad
 \Longleftrightarrow \quad \omega(\nu_{x_2},\nu_{x_1})=0,
\end{equation*}
as claimed.
\end{proof}
\begin{lemma}\label{omega_vanishing_connected_set}
The sets $\P=\{(x_1,x_2)\in M\times M\mid \omega(\nu_{x_1},\nu_{x_2})>0\}$ and $\{(x_1,x_2)\in M\times M\mid \omega(\nu_{x_1},\nu_{x_2})<0\}$ are connected.
\end{lemma}

\begin{proof}Consider the Gauss map $M\ni x\mapsto \nu_x\in S^{2n-1}$. Then $\P$ is mapped to
\begin{equation*}
\P_0:=\{(a,b)\in S^{2n-1}\times S^{2n-1}\mid \omega(a,b)>0\}.
\end{equation*}
Clearly $b\neq\pm a$. Therefore, the simple interpolation
\begin{equation*}
a_t:=\frac{a+t(b-a)}{||a+t(b-a)||}\in S^{2n-1}, \; t\in(0,1]
\end{equation*}
and 
\begin{equation*}
\omega(a,a_t)=\frac{\omega(a,a+t(b-a))}{||a+t(b-a)||}=\frac{t\omega(a,b)}{||a+t(b-a)||}
\end{equation*}
shows that we can move the second factor close to $a$. That is, we can move any point in  $\P_0$ into a tubular neighborhood of the diagonal which is connected. The same argument works for $<$. Since the Gauss map is a bijection the results follows.
\end{proof}
We define the map $\Phi:\P\to\P$ geometrically similarly to the procedure from Section \ref{refllaw}, see Figure \ref{higher_d}.
\begin{figure}[hbtp] 
\centering
\includegraphics[width=3in]{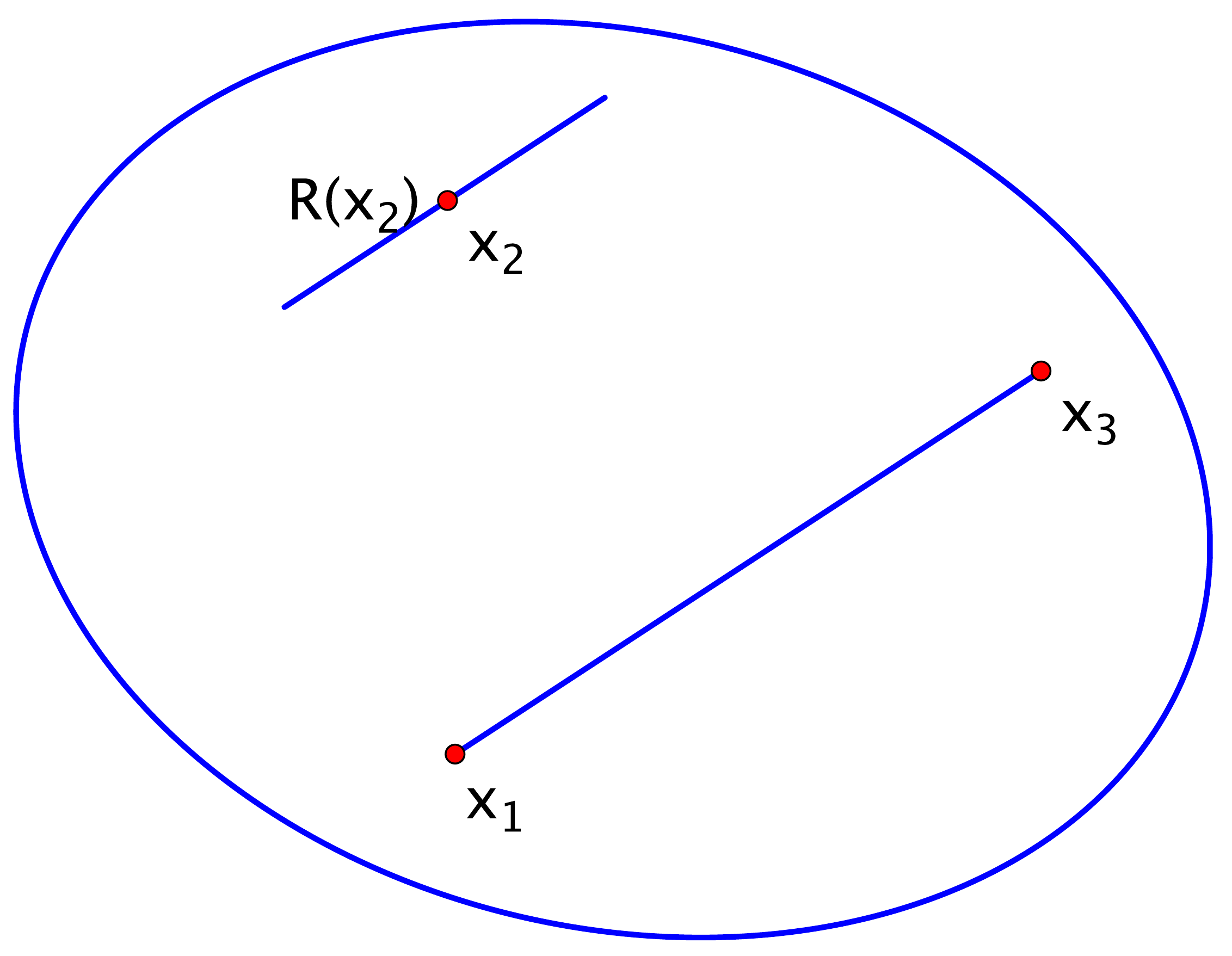}
\caption{Symplectic billiard map in higher dimensions.}
\label{higher_d}
\end{figure}
\begin{lemma}\label{lem:basic_def_sympl_case}
Given $(x_1,x_2)\in \P$, there exists a unique point $x_3\in M$ with
\begin{equation*}
\big(x_1+R(x_2)\big)\cap M =\{x_1,x_3\}. 
\end{equation*} 
Moreover, $(x_2,x_3)\in\P$ and
\begin{equation}\nonumber
x_3-x_1=tJ\nu_{x_2}
\end{equation}
with $t>0$.
\end{lemma}

\begin{remark}
{\rm Using the language of contact geometry one, of course, realizes that $J\nu$ is positively proportional to the Reeb vector field $\vec{R}$ of the standard contact form $\alpha:=-dr\circ J$ on $M$.  Here $r$ denotes the radial coordinate on $\R^{2n}$. In particular, $\vec{R}$ orients (or trivializes) the line field $R$ and the definition of the positive part of the phase space corresponds to conforming to this orientation.
} 
\end{remark}

\begin{proof}Since $M$ is convex, $\big(x_1+R(x_2)\big)\cap M =\{x_1,x_3\}$ with possibly $x_1=x_3$. Since $(x_1,x_2)\in\P$, we know that $R(x_2)\not\subset T_{x_1}M$, and therefore, $x_1\neq x_3$. 

To show that $(x_2,x_3)\in\P$, we argue, as in the 2-dimensional case, by continuity. The statement is clearly true if $x_2$ is close to $x_1$. If $\omega(\nu_{x_2},\nu_{x_3})\leq0$ then, using the connectedness of $\{(x_1,x_2)\in M\times M\mid \omega(\nu_{x_1},\nu_{x_2})<0\}$, we can move the point $x_2$ in order to achieve $\omega(\nu_{x_2},\nu_{x_3})=0$. The latter is equivalent to $R(x_2)\subset T_{x_3}M$ which contradicts $x_1\neq x_3$.

To determine the sign in $x_3-x_1=tJ\nu_{x_2}$ consider the set $M_x^+:=\big\{y\in M\mid \omega(\nu_x,\nu_y)>0\big\}$ and similarly $M_x^-$. Both are connected since they correspond to hemispheres under the Gauss map. We can rephrase the definition of $\P$ as 
\begin{equation}\nonumber
(x_1,x_2)\in\P \quad \Longleftrightarrow \quad x_1\in M_{x_2}^- \quad \Longleftrightarrow \quad x_2\in M_{x_1}^+,
\end{equation}
and the above discussion as
\begin{equation}
x_3\in M_{x_2}^+\;.
\end{equation}\nonumber
Since $x_1,x_2$ determine $x_3$ we write
\begin{equation}\nonumber
x_3-x_1=t(x_2)J\nu_{x_2},
\end{equation}
where we think of $x_1$ as fixed and $x_2$ as a variable. 

First we claim that the sign of $t(x_2)$ does not depend on $x_2$ as long as $(x_1,x_2)\in\P$, i.e., $x_2\in M_{x_1}^+$. Assume that $t(x_2)$ changes sign or vanishes. Since $M_{x_1}^+$ is connected, we can always reduce to the latter case, i.e., we find a point $x_2$ with $t(x_2)=0$. But this means $x_3=x_1$, which directly contradicts the very first assertion of this proof. 

Since the sign of $t(x_2)$ is fixed, we now compute it at a convenient point. For that, let $x_2$ be the point where $\nu_{x_2}=aJ\nu_{x_1}$ with $a>0$. This point exists and is unique again by the Gauss map being a diffeomorphism. It follows that $(x_1,x_2)\in\P$ since
\begin{equation}\nonumber
\omega(\nu_{x_1},\nu_{x_2})=\omega(\nu_{x_1},aJ\nu_{x_1})= a\cdot|\nu_{x_1}|^2=a>0.
\end{equation}
Finally we observe that
\begin{equation}\nonumber
x_3-x_1=t(x_2)J\nu_{x_2}=at(x_2)JJ\nu_{x_1}=-at(x_2)\nu_{x_1}
\end{equation}
which, again by convexity, forces $t(x_2)>0$, as $\nu$ is the outward normal.
\end{proof}
By Lemma \ref{lem:basic_def_sympl_case}, we can again define a map
\begin{equation*}
\Phi:\P\to\P,\
(x_1,x_2)\mapsto \Phi(x_1,x_2):=(x_2,x_3)
\end{equation*}
by the above rule. Now it follows that $S(x_1,x_2)=\omega(x_1,x_2)$ is a generating function, indeed:
\begin{equation}\label{def_phi}
(x_2,x_3)=\Phi(x_1,x_2)\quad \Longleftrightarrow \quad \frac{\partial}{\partial x_2}\big(S(x_1,x_2)+S(x_2,x_3)\big)=0,
\end{equation}
by the latter we, of course, mean that 
\begin{equation*}
\begin{aligned}
\omega(x_1,v)+\omega(v,x_3)=0\quad\forall v\in T_{x_2}M.
\end{aligned}
\end{equation*}
Since $x_1\neq x_3$, this is equivalent to the condition 
$
x_3-x_1\in R(x_2),
$
as needed.

We can again extend the map $\Phi$ continuously to 
$$\bar\P:=\{(x_1,x_2)\in M\times M\mid \omega(\nu_{x_1},\nu_{x_2})\geq0\}$$
 by
\begin{equation}
\begin{aligned}
&\Phi(x,x):=(x,x)\\
&\omega(\nu_{x_1},\nu_{x_2})=0\;\;\Rightarrow\;\; \Phi(x_1,x_2):=(x_2,x_1).
\end{aligned}
\end{equation}
Equation \eqref{def_phi} implies that the map $\Phi$ preserves the closed 2-form 
\begin{equation*}
\Omega_{(x_1,x_2)}:= \frac{\partial^2 S}{\partial x_1\partial x_2}(x_1,x_2)\ dx_1\wedge dx_2\in \Omega^2(M\times M).
\end{equation*}
That is, for $(u_i,v_i)\in T_{(x,y)}M\times M$,
\begin{equation*}
\Omega_{(x,y)}\big((u_1,v_1),(u_2,v_2)\big)=\omega(u_1,v_2)-\omega(u_2,v_1).
\end{equation*}

\begin{lemma}
For $x,y\in M$, we have
\begin{equation*}
\Phi(x,y)=(y,x)\quad \Longleftrightarrow \quad \omega(\nu_x,\nu_y)=0 \quad \Longleftrightarrow \quad \Omega_{(x,y)} \text{ is degenerate}.
\end{equation*} 
\end{lemma}

\begin{proof}
For $(u_1,v_1)\neq0$ we have
\begin{equation}
\begin{aligned}
\Omega_{(x,y)}&\big((u_1,v_1),(u_2,v_2)\big)=0\quad\forall (u_2,v_2)\in T_{(x,y)}M\times M\\
&\Longleftrightarrow \quad \omega(u_2,v_1)=0\;\forall u_2\in T_xM\text{ and }\omega(u_1,v_2)=0\;\forall v_2\in T_yM\\
&\Longleftrightarrow \quad R(x)\subset T_yM \text{ or/and } R(y)\subset T_xM\\
&\Longleftrightarrow \quad \omega(\nu_x,\nu_y)=0\\
&\Longleftrightarrow \quad\Phi(x,y)=(y,x),\\
\end{aligned}
\end{equation}
as claimed. We recall from Lemma \ref{omega_0_equiv_charc_direction_in_tangent_space} that the relation $R(x)\subset T_yM$ is symmetric in $x$ and $y$. Therefore ``or'' and ``and'' in the third line are equivalent.
\end{proof}
\begin{remark}$ $
{\rm
\begin{enumerate} 
\item As in the planar case,t strict convexity gives rise to an involution $M\ni z\mapsto z^*\in M$ characterized by $R(z)=R(z^*)$.
\item  Similarly to polygons in the plane, one can define symplectic billiards in convex polyhedra in symplectic space. For example, it would be interesting to study a simplex in $\C^2$.
\end{enumerate}
}
\end{remark}

\subsubsection{Continuous limit} \label{contlim}

Consider a (normal) billiard trajectory inside a convex hypersurface $M$ that makes very small angles with the hypersurface (a grazing trajectory); as the angles tend to zero, one expects such a trajectory to have a geodesic on $M$ as a limit. One
observes this at the level of generating functions: the chord length $|xy|,\ x,y\in M$, becomes, in the limit $y \to x$, the  function $L(x,v) = |v|$ of the tangent vectors $v \in T_x M$. The extremals of the Lagrangian $L$ are non-parameterized geodesics on $M$.

What happens in the continuous limit with the symplectic billiard? 

Geometrically, one expects the following. Let $Z_0,Z_1,Z_2,Z_3,\ldots$ be an orbit. As the points merge together, the direction of the segment $Z_0 Z_2$ tends to the characteristic direction $R(Z_1)$, and the direction of $Z_1 Z_3$ to $R(Z_2)$. Therefore, in the limit, the even-numbered points lie on a characteristic line, and so do the odd-numbered points; and in the limit, the two characteristic lines merge together. Note however that the direction $Z_0 Z_1$ is not necessarily characteristic, and  the limiting behavior of the directions $Z_j Z_{j+1}$ is not determined.

One observes the same phenomenon at the level of generating functions. The continuous limit of the generating function $\omega(x,y),\ x,y \in M$, is the Lagrangian $L(x,v)=\omega(x,v)$ on the tangent bundle $TM$. The Euler-Lagrange equation with constraints reads
$$
\frac{d L_v}{dt} - L_x = \lambda \nu,
$$
where $t$ is time, $\nu$ is a normal, and $\lambda$ is a Lagrange multiplier. Since $L(x,v)=\langle Jx,v\rangle$, one has
$L_v = Jx, L_x = -Jv$, and the Euler-Lagrange equation reduces to $v = \mu(x,v) J(\nu)$. This is the equation of non-parameterized characteristic line on $M$.

Note however that, due to the degeneracy of the Lagrangian $L$, we end up with a first order differential equation, rather than the second order one.

Perhaps a better way to take a continuous limit is to treat the even-numbered and odd-numbered vertices of a symplectic billiard trajectory separately: points $Z_0,Z_2,\ldots$ converge to one curve $\g_0(t) \subset M$, and points $Z_1,Z_3,\ldots$ to another curve $\g_1(t)\subset M$. In this limit, one obtains the following relation between the curves:
\begin{equation} \label{double}
R(\g_1(t)) \sim \g_0'(t),\ R(\g_0(t)) \sim \g_1'(t),
\end{equation}
where $ a\sim b$ means that the vectors $a$ and $b$ are proportional. Compare to the discussion for symplectic ellipsoids in Section \ref{ellip}.

The problem of finding pairs of curves on $M$ satisfying (\ref{double}) is variational.

\begin{lemma} \label{pairvar}
A pair of closed curves $(\g_0(t),\g_1(t))$ on $M \subset \C^n$, satisfying (\ref{double}), is critical for the functional 
$$
L(\g_0,\g_1)= \int \omega(\g_0'(t),\g_1(t))\ dt.
$$
\end{lemma} 

\begin{proof} 
Integrating by parts, we see that $L(\g_0,\g_1) = L (\g_1,\g_0)$. Consider a variation of $\g_1$ given by the vector field $v(t)$ on $M$ along $\g_1$. Having
$$
\int \omega(\g_0'(t),v(t))\ dt =0
$$
for all such $v$ is equivalent to $\g_0'(t)$ being symplectically orthogonal to $T_{\g_1(t)}M$ for all $t$, that is, $\g_0'(t) \sim R(\g_1(t))$. Due to symmetry of the functional $L$, the same argument yields $\g_1'(t) \sim R(\g_0(t))$.
\end{proof}
It would be interesting to describe pairs of curves on $M$ satisfying (\ref{double}). A particular solution is $\g_0=\g_1$, a characteristic curve. If $M$ is a unit sphere, then a pair $(\g_0,\g_1)$ of geodesic circles related by $\g_1 = J(\g_0)$ are all other solutions.

\subsection{Ellipsoids} \label{ellip}

\subsubsection{Symplectic billiard and the usual billiard} \label{ssell}

It is well known that the billiard ball map in an ellipsoid is completely integrable, see, e.g., \cite{Ta}. In this section we describe a close relation between the symplectic and the usual billiard in ellipsoids that, in particular, implies complete integrability of the former.

Consider an ellipsoid in $\R^{2n}$ with Darboux coordinates  $x_1,\ldots, x_n, y_1, \ldots, y_n$ and the symplectic structure $\omega_0 = \sum dx_j \wedge dy_j$. Applying a linear symplectic transformation and homothety, we may assume that the ellipsoid is given by the equation
$$
\frac{x_1^2+y_1^2}{a_1} + \frac{x_2^2+y_2^2}{a_2} + \ldots + \frac{x_n^2+y_n^2}{a_n} =1,\qquad a_1,\ldots, a_n>0.
$$
The diagonal linear transformation 
$$
x_j \mapsto \sqrt{a_j} x_j, y_j \mapsto \sqrt{a_j} y_j,\ j=1,\ldots,n,
$$
 takes the ellipsoid to the unit sphere and transforms the symplectic form to $\omega = \sum a_j dx_j \wedge dy_j$. We shall consider the symplectic billiard inside the unit sphere $S^{2n-1}$ defined by this symplectic form. The characteristic direction at the point 
$z\in S^{2n-1}$ is $\R\cdot R(z)$, where the complex linear operator $R:\C^n \to \C^n$ is diagonal with the entries $ia_1^{-1},\ldots, i a_n^{-1}$.

Consider the linear map $R^{-1}: \C^n \to \C^n$. It takes the unit sphere to the ellipsoid $E$ given by the equation
\begin{equation} \label{newell}
 \frac{|w_1|^2}{a_1^2} + \frac{|w_2|^2}{a_2^2} + \ldots + \frac{|w_n|^2}{a_n^2} =1,
\end{equation}
where $w_1,\ldots,w_n$ are complex coordinates in the target space.

\begin{theorem} \label{hodo}
Let $(\ldots, Z_0, Z_1, Z_2,\ldots)$,  be a trajectory of the symplectic billiard map in the unit sphere with respect to the symplectic form $\omega = \sum a_j dx_j \wedge dy_j$. Then a sequence $(\ldots, R^{-1}(Z_0), R^{-1} (Z_2), R^{-1} (Z_4), \ldots)$ is a billiard trajectory in $E$. Conversely, to a billiard trajectory $(\ldots, W_0, W_2, W_4, \ldots)$ in $E$ there corresponds a unique symplectic billiard trajectory $(\ldots, Z_0, Z_1, Z_2,\ldots)$ in $S^{2n-1}$ with $Z_0 = R(W_0), Z_2 = R(W_2)$, etc.
\end{theorem}

\begin{proof}Consider the points $Z_0, Z_2, Z_4$ of a symplectic billiard trajectory. The points $Z_1$ and $Z_3$ are uniquely determined by the symplectic billiard reflection law:
$$
R(Z_1) = t_1 (Z_2-Z_0), R(Z_3) = t_3 (Z_4-Z_2), t_1 > 0, t_3 > 0.
$$
Hence $Z_1 = t_1 R^{-1}(Z_2-Z_0)$, and the normalization $|Z_1|^2=1$ uniquely determines $t_1$; likewise for $t_3$. Thus
$$
Z_1 = \frac{R^{-1}(Z_2-Z_0)}{|R^{-1}(Z_2-Z_0)|},\  Z_3 = \frac{R^{-1}(Z_4-Z_2)}{|R^{-1}(Z_4-Z_2)|}.
$$
The symplectic billiard reflection law also implies that
$$
R(Z_2) = t (Z_3-Z_1) = t \left(\frac{R^{-1}(Z_4-Z_2)}{|R^{-1}(Z_4-Z_2)|} + \frac{R^{-1}(Z_0-Z_2)}{|R^{-1}(Z_0-Z_2)|}\right). 
$$
Set $R^{-1} (Z_j)=W_j,\ j=0,2,4$, and rewrite the last equation as
\begin{equation} \label{usrefl}
R^2 (W_2) =t \left(\frac{W_0 - W_2}{|W_0 - W_2|} + \frac{W_4 - W_2}{|W_4 - W_2|}\right).
\end{equation}
Note that the vector $R^2 (W_2)$ is normal to the ellipsoid $E$ given by (\ref{newell}). Therefore equation (\ref{usrefl}) describes the billiard reflection in $E$ at point $W_2$ that takes  $W_0 W_2$ to $W_2 W_4$, as claimed.

Conversely, given a segment of a billiard trajectory $W_0,W_2,W_4$ in $E$, one defines $Z_j = R(W_j),\ j=0,2,4$, 
$$
Z_1=\frac{W_2-W_0}{|W_2-W_0|},\ Z_3=\frac{W_4-W_2}{|W_4-W_2|}.
$$
Then equation (\ref{usrefl}) implies that $Z_0,\ldots, Z_4$ is a segment of a symplectic billiard trajectory. 
\end{proof}
\subsubsection{Continuous version} \label{contvers}

Let us also present a continuous version of Theorem \ref{hodo}.

\begin{proposition} \label{conhodo}
 Let $A = {\rm diag} (a_1,a_2,\ldots,a_n)$ be a diagonal matrix with real positive entries, and let
 $x(t)$ be a characteristic curve on the ellipsoid $\langle Ax,x\rangle=1$. The linear map $A^{-1/2}: \C^n \to \C^n$ takes this curve to a geodesic curve on the ellipsoid $\langle A^2 y,y\rangle=1$.
\end{proposition}

\begin{proof}If $y=A^{-1/2}x$ and $\langle Ax,x\rangle=1$, then $\langle A^2 y,y\rangle=1$. Note that $Ax$ is a normal to the ellipsoid $\langle Ax,x\rangle=1$ at the point $x$.
If $x(t)$ is a characteristic, then $x' = f JA x$ where $f(t)$ is a non-vanishing function. The matricies $A$ and $J$ commute. Differentiate:
$$
x'' =  f' JA x + f JA x' = \frac{f'}{f} x' - f^2 A^2 x,
$$
hence
$$
y'' = \frac{f'}{f} y' - f^2 A^2 y.
$$
Since $A^2 y$ is a normal to the ellipsoid $\langle A^2 y,y\rangle=1$ at the point $y$, we conclude that $y'' \in {\rm Span} (y',\nu_y)$, that is, $y(t)$ is a geodesic.
\end{proof}
One can reverse the argument:  start with a geodesic on the ellipsoid $\langle A^2 y,y\rangle=1$ and construct a pair of curves on the ellipsoid $\langle Ax,x\rangle=1$ in the relation (\ref{double}). We use the same notations as before.

\begin{proposition} \label{reverse}
Let $y(t)$ be a geodesic on the ellipsoid $\langle A^2 y,y\rangle=1$ satisfying 
$$
y''(t) = \frac{f'(t)}{f(t)} y'(t) +g(t) A^2 y(t),
$$
where $f$ and $g$ are some functions with $f(t)\neq0$ for all $t$. Set 
$$
x(t)=A^{1/2} y(t),\  \bar x(t)= \frac{c}{f(t)} J A^{-1} x'(t).
$$
Then, for a suitable choice of the constant $c$, the pair of curves $(x,\bar x)$ lie on the ellipsoid $\langle Ax,x\rangle=1$ and satisfy (\ref{double}).
\end{proposition}

\begin{proof}We know from the previous proof that $x(t)$ lies on the ellipsoid $\langle Ax,x\rangle=1$. Also
\begin{equation} \label{geoeq}
x''(t) = \frac{f'(t)}{f(t)} x'(t) +g(t) A^2 x(t).
\end{equation} 
Let $h(t)=\langle x'(t),A^{-1} x'(t)\rangle$. Then 
$$
h' = 2 \langle x'',A^{-1} x'\rangle = 2 \frac{f'}{f} \langle x',A^{-1} x'\rangle + 2g \langle A^2 x,x'\rangle = 2  \frac{f'}{f} h,
$$
where we used (\ref{geoeq}) and the fact that $Ax$ is orthogonal to $x'$. It follows that $h=C f^2$ where $C$ is a constant. Then
$$
\langle A \bar x, \bar x \rangle = \frac{c^2}{f^2(t)}   \langle x'(t),A^{-1} x'(t)\rangle = c^2 C,
$$
hence, if $c=C^{-1/2}$, we have $\langle A \bar x, \bar x \rangle =1$.

It remains to check (\ref{double}). That $R(\bar x) \sim x$ is clear, since $R = JA$. Finally,
$$
\bar x' = -\frac{cf'}{f^2} J A^{-1} x' + \frac{c}{f(t)} J A^{-1} x'' = -\frac{cf'}{f^2} J A^{-1} x' + \frac{c}{f} J A^{-1} 
\left( \frac{f'}{f} x' +g A^2 x \right) \sim JAx,
$$
as needed.
\end{proof}
Proposition \ref{reverse} seems to indicate that the correct continuous limit of symplectic billiard is a pair of curves satisfying (\ref{double}), as discussed in Section \ref{contlim}.

\begin{remark}
{\rm The curve $\bar y = A^{-1/2} \bar x$ is also a geodesic on the ellipsoid $\langle A^2 y,y\rangle=1$. The  geodesics $y$ and $\bar y$ are related by the composition of $J$ and the skew-hodograph transformation, see section 3 in \cite{MV}. 
}
\end{remark}

\subsubsection{Symplectic billiards and the discrete Neumann system} \label{Neumann}

The discrete Neumann system $(Z_0,Z_1) \mapsto (Z_1,Z_2)$ is a Lagrangian map on the Cartesian square of the unit sphere given by the equation
$$
Z_0 + Z_2 = \lambda A(Z_1),
$$
where $A$ is a self-adjoint linear map and $\lambda$ is a factor  determined by the normalization $|Z_2|=1$, see \cite{MV,Ve1,Ve2}. Symplectic billiard inside the unit sphere is given by a similar equation
$$
Z_2 - Z_0 = t R(Z_1),
$$
where $R$ is an anti self-adjoint linear map and $t$ is a suitable factor. 

Our Theorem \ref{hodo} is an analog of Theorem 6 in \cite{MV} that relates, in a similar way, the discrete Neumann system and the billiard inside an ellipsoid. Thus, similarly to the Neumann system, the symplectic billiard is ``a square root" of the billiard system in an ellipsoid (cf. \cite{Ve2}). Note however that the latter ellipsoid is not generic: its axes are equal pairwise.

\subsubsection{Integrals} \label{ellint}

The symplectic billiard map possesses  a collection of particularly simple integrals described in the following proposition.
For a point $Z_j \in S^{2n-1} \subset \R^{2n}$, write $Z_j=(x_{1j},y_{1j},\ldots,x_{nj},y_{nj})$.

\begin{proposition} \label{easyint}
The following functions are integrals of the symplectic billiard map $(Z_0,Z_1) \mapsto (Z_1,Z_2)$:
$$
I_k (Z_0,Z_1) = x_{k0} x_{k1} + y_{k0} y_{k1},\ k=1,\ldots,n,\ \  {\rm and}\ \ J(Z_0,Z_1) = \langle R(Z_0), Z_1\rangle.
$$
\end{proposition}

\begin{proof}One has
$$
Z_2 = Z_0 + \frac{2 \langle R(Z_0),Z_1\rangle}{|R(Z_1)|^2} R(Z_1).
$$
Substitute this $Z_2$ to $J(Z_1,Z_2)$ and simplify to obtain $J(Z_0,Z_1)$.

Likewise, to show that $I_k$ is an integral it suffices to check that $I_k(R(Z_1),Z_1) =0$. This follows from the fact that $R$ is a diagonal map of $\C^n$ that, up to a factor, multiplies each  coordinate by $i$. 
\end{proof}
In particular, a trajectory $(\ldots, Z_0, Z_1, Z_2, \ldots)$ of the symplectic billiard in a sphere is an equilateral polygonal line.

The integrals $I_k$ occur due to the symmetry of the ellipsoid (or, equivalently, of the operator $R$): $I_k$ corresponds to the rotational symmetry in $k$th coordinate complex line via E. Noether's theorem. 

 Let us explain the billiard origin of the integral $J$.

Consider the billiard system in an ellipsoid $E$ given by the equation $\langle Aw,w\rangle = 1$. The  phase space of the billiard consists of the tangent vectors $(w,v)$ with the foot point $w\in E$ and a unit inward vector $v$. It is known that the function $\langle Aw,v\rangle$ is an integral of the billiard transformation, see, e.g., \cite{Ta}.

Consider the pull-back of the integral $\langle Aw,v\rangle$ under the map $R^{-1}: \C^n \to \C^n$. This is a function of $(Z_0,Z_2)$, and since $Z_2$ is determined by $Z_0$ and $Z_1$ via the symplectic billiard reflection, this is also a function of $(Z_0,Z_1)$, a phase point of the symplectic billiard.

\begin{proposition} \label{intJ}
This function equals $\langle Z_0, R(Z_1)\rangle$.
\end{proposition}

\begin{proof} 
One has $R^{-1} (Z_0) = w$, that is, $R(w) = Z_0$. Likewise, $R(v)$ is positive-proportional to $Z_2-Z_0$ which, by the symplectic billiard reflection law, is positive-proportional to $R(Z_1)$. Since both $v$ and $Z_1$ are unit, $v=Z_1$. Since $A= -R^2$ and $R$ is anti self-adjoint, we have
$$
\langle Aw,v\rangle = \langle Rw,Rv\rangle = \langle Z_0,R(Z_1)\rangle,
$$
as claimed.
\end{proof}
\subsubsection{Low-period orbits} \label{lowper} 

Let us also mention a property of low-period orbits of the symplectic billiard in an ellipsoid. 

By a coordinate subspace in $\C^n$ we mean a subspace spanned by any number of the complex coordinate lines. As before, we consider symplectic billiard inside the unit sphere, with the characteristic vector  given by a diagonal complex linear operator $R$ with the entries $ia_1^{-1},\ldots, i a_n^{-1}$. Assume that $R$ is generic in the sense that $a_1<a_2<\ldots < a_n$.

\begin{proposition}\label{small}
Let $Z_1,\ldots,Z_k$ be a $k$-periodic symplectic billiard orbit with $k<2n$. If $k$ is even, then the orbit is contained in a coordinate subspace of dimension at most $k$, and if $k$ is odd, in a coordinate subspace of dimension at most $k-1$.
\end{proposition}

\begin{proof}Let $L$ be the vector space spanned by $Z_j,\ j=1,\ldots,k$. The law of the symplectic billiard reflection implies that $R(Z_j)$ is proportional to $Z_{j+1}-Z_{j-1}$ for all $j=1,\ldots,k$ (the indices are understood cyclically mod $k$). Therefore $L$ is an invariant subspace of the linear map $R$.

We claim that $L$ is a coordinate subspace. Let $\C_1,\ldots,\C_n$ be the complex coordinate lines, and let $\pi_j$ be the projection of $\C^n$ on $\C_j$. 

Consider the projections $\pi_j(L)$. If $\pi_j(L)=0$ for some $j$, then we may ignore the $j$th coordinate in what follows. In other words, assume that $\pi_j(L)\neq 0$ for all $j$. The claim now is that $L$ is the whole space $\C^n$.

Let $Z=(1,z_2, \ldots,z_n) \in L$. Then $a_1^{4m} R^{4m}(Z) \in L$. In the limit $m\to\infty$, we obtain the basic vector $(1,0,\ldots,0)$, hence $\C_1 \subset L$. Factorize by $\C_1$ and repeat the argument. It implies that $\C_1 + \C_2 \subset L$, and so on. Thus $L = \C^n$, as claimed. 

To finish the proof, note that a coordinate subspace is always even-dimensional.
\end{proof}
Proposition \ref{small} is in perfect agreement with Proposition  7.16 of \cite{DR}: ``if the billiard trajectory within an ellipsoid in $d$-dimensional Euclidean space is periodic with period $n \le d$, then it is placed in one of the $(n - 1)$-dimensional planes of symmetry of the ellipsoid".

\subsubsection{Round sphere} \label{round}

The case of $a_1=a_2=\ldots =a_n=1$, that is, the case when $M=S^{2n-1}$ is the unit sphere, is special. In this section, we describe this case in further detail. 

Let $z_0 z_1$ be the initial segment of a billiard trajectory, $z_0,z_1 \in S^{2n-1}$. The symplectic billiard map $\Phi$ is given by the formula
\begin{equation} \label{form}
z_2 = z_0 + 2 \omega (z_0,z_1) J(z_1),
\end{equation}
where $J$ is multiplication by $i$. The quantity $\omega (z_0,z_1)$ is an integral: $\omega (z_0,z_1) = \omega (z_1,z_2)$, and we denote it simply by $\omega$.

In the phase space $\P$ we have $0 \le \omega$, and also $\omega \le 1$ because $M$ is the unit sphere. Set $\omega = \sin \alpha$ for $0\leq \alpha \leq \pi/2$,
and let 
\begin{equation} \label{roots}
\lambda_1 = e^{i\alpha},\ \lambda_2 = - e^{-i\alpha}.
\end{equation}

The case $\omega=0$ is special: this is the boundary of the phase space, and the orbits are 2-periodic.
The case $\omega=1$ is also special. In this case, the orbit of $\Phi$ is 4-periodic. Indeed, if $\omega (z_0,z_1)=1$ then $z_1=J(z_0)$ and hence one obtains the sequence of points 
$$
z_0 \;\;\;\mapsto\;\;\; z_1= J(z_0) \;\;\;\mapsto\;\;\; z_2= -z_0 \;\;\;\mapsto\;\;\; z_3=- J(z_0) \;\;\;\mapsto\;\;\; z_4=z_0 \;\;\;\mapsto\;\;\; \ldots
$$
The general case is described in the next proposition, where we assume that $\omega < 1$.

\begin{proposition} \label{superint}
One has
\begin{equation} \label{explicit}
z_n = \frac{\lambda_1^{n-1} - \lambda_2^{n-1}}{\lambda_1 - \lambda_2} z_0 + 
\frac{\lambda_1^{n} - \lambda_2^{n}}{\lambda_1 - \lambda_2} z_1.
\end{equation}
The $\Phi$-orbit of a point lies on the union of two circles. The orbit is periodic if $\alpha$ is $\pi$-rational and dense on the two circles otherwise. If $\alpha= 2\pi (p/q)$, where $p/q$ is in the lowest terms, then the period equals $q$ for even $q$, and $2q$ for odd $q$.
\end{proposition}

\begin{proof}Equation (\ref{form}) is a second order linear recurrence with constant coefficients that generates the sequence $z_0,z_1,z_2,\ldots$ Its solution is a linear combination of two geometric progressions whose denominators are the roots of the characteristic equation $\lambda^2 - 2i\omega \lambda -1=0$. These roots are distinct, and they are given by the formula
$$
\lambda_{1,2} = i \omega \pm \sqrt{1-\omega^2},
$$
coinciding with  (\ref{roots}). Choosing the coefficients of the two geometric progressions to satisfy the initial conditions, one obtains formula (\ref{explicit}). 

The map $\Phi: (z_0,z_1) \mapsto (z_1,z_2)$ is a complex linear self-map of $\C^{2n}$, it has the eigen-values $\lambda_1$ and $\lambda_2$, each with multiplicity $n$. Writing $(z_0,z_1)$ as a column vector, $\Phi$ has the matrix
$$
\left(
\begin{matrix}
0&E\\
E & 2i\omega E\\
\end{matrix}
\right),
 $$
where each entry is an $n\times n$ block.

Decompose $\C^{2n}$ into the direct sum of the eigen-spaces. Specifically, 
$$
\left(
\begin{matrix}
a \\
b  \\
\end{matrix}
 \right) = 
 \frac{1}{\lambda_2-\lambda_1} \left(
\begin{matrix}
b - \lambda_1 a \\
\lambda_2 b + a  \\
\end{matrix}
 \right) +
 \frac{1}{\lambda_1-\lambda_2} \left(
\begin{matrix}
b - \lambda_2 a \\
\lambda_1 b + a  \\
\end{matrix}
 \right),
$$
where $a$ and $b$ are column vectors in $\C^n$.

Writing a vector accordingly as $(u,v)$, one has
\begin{equation} \label{Phieq}
\Phi: (u,v) \mapsto (e^{i\alpha} u , - e^{-i\alpha} v).
\end{equation} 
The orbit of $(u,v)$ lies on the union of two circles $(e^{it} u ,  e^{-it} v)$ and $(e^{it} u , - e^{-it} v)$, where $t\in \R$. The orbit is finite if $\alpha$ is $\pi$-rational, and dense on the two circles otherwise. Let
$$
\alpha = 2\pi \frac{p}{q},
$$
where $p$ and $q$ are coprime. Assume that $v\neq 0$. If $q$ is even then the orbit closes up after $q$ iterations, but if $q$ is odd, one needs twice as many, due to the alternating sign of the second component in (\ref{Phieq}). 
\end{proof}

\subsection{Periodic orbits} \label{periodic}

Let $M \subset \R^{2n}$ be a smooth, strictly convex, closed hypersurface. In this section we discuss periodic trajectories of the symplectic billiard map in $M$. Given a $k$-periodic trajectory, one can cyclically permute its vertices or reverse their order; accordingly, we  count the orbits of this dihedral group $D_k$ action.

\subsubsection{Existence of periodic orbits of any period} \label{exist}

\begin{theorem} \label{weak}
For every $k\ge 2$, the symplectic billiard map has a $k$-periodic trajectory.
\end{theorem}

If $k$ is not prime, we do not exclude the case that this trajectory may be multiple, that is, a lower-periodic trajectory, traversed several times.

\begin{proof}
A periodic trajectory is a critical point of the symplectic area function $F(z_1,\ldots,z_k)=\sum_{i=1}^k \omega(z_i,z_{i+1})$ on $k$-gons inscribed in $M$. Due to compactness, this function attains maximum. Let us show that the respective critical point is a genuine periodic trajectory of the symplectic billiard map.

Let $z_1,\ldots,z_k$ be a critical polygon for $F(z_1,\ldots,z_k)=\sum_{i=1}^k \omega(z_i,z_{i+1})$. Then
$
\omega(z_{i-1},v)+\omega(v,z_{i+1}) =0\ {\rm for\ all}\ v\in T_{z_i} M,\ i=1,\ldots,n,
$
that is, $\omega(v,z_{i+1}-z_{i-1}) =0$. 
If $z_{i+1} \neq z_{i-1}$ then $z_{i+1}-z_{i-1} \in R(z_i)$, as the symplectic billiard map requires.
 A problem arises if $z_{i+1} = z_{i-1}$.

Let us show that if $z_{i+1} = z_{i-1}$ then the value of $F(z_1,\ldots,z_k)$ is not maximal. Indeed, in this case the two terms, $\omega(z_{i-1},z_i)$ and $\omega(z_i,z_{i+1})$ cancel each other, and the function $F$ is the symplectic area of the $(k-2)$-gon $z_1,\ldots,z_{i-1},z_{i+2},\ldots,z_k$. Thus it suffices to show that the maximum of the symplectic area $F$ is attained on non-degenerate $k$-gons (and not on polygons with fewer sides).

To do this we show that, given an inscribed polygon $P$, one can add to it one vertex  (and thus two vertices) so that the symplectic area  increases. Indeed, let $ac$ be a side of $P$. We want to find a point $b\in M$ so that $\omega (a,b) + \omega(b,c) > \omega (a,c)$. This is equivalent to saying that the symplectic area of the triangle $abc$ is positive. To achieve this, take an affine symplectic plane through $ac$ and choose point $b$ appropriately on its  intersection curve with $M$.
\end{proof}

\subsubsection{Periods three and four} \label{threeand4}

The result of Theorem \ref{weak} is quite weak: we believe, the actual number of periodic orbits is much larger (see, e.g., \cite{FT} for the usual multi-dimensional billiards). The next theorem concerns small periods.

\begin{theorem} \label{smallper}
For every $M$ as above, the number of 3-periodic symplectic billiard trajectories is not less than $2n$.  The same lower bound holds for the number of 4-periodic trajectories.
\end{theorem}

\begin{proof}The case $k=3$ is contained in \cite{Ta2} where 3-periodic trajectories of the outer billiard are studied. The function whose critical points are these 3-periodic trajectories is the same: it is the symplectic area of an inscribed triangle. This is also clear geometrically: if $ABC$ is a 3-periodic trajectory of the outer billiard, then the midpoints of the sides of the triangle $ABC$ form a 3-periodic trajectory of the symplectic billiard. 

Let us consider the case $k=4$. 

Let $z_1,z_2,z_3,z_4$ be a 4-periodic orbit. Then $z_3-z_1 \in R(z_2)$ and $z_3-z_1 \in R(z_4)$, and likewise, $z_4-z_2 \in R(z_3)$ and $z_4-z_2 \in R(z_1)$. It follows that $R(z_2)=R(z_4)$ and $R(z_1) = R(z_3)$. Using strict convexity of $M$, we conclude that $z_4=z_2^*$ and $z_3=z_1^*$, where the involution $z \mapsto z^*$ is as before: the tangent hyperplanes to $M$ at $z$ and $z^*$ are parallel.

The (oriented) chords $z^* z$ are called affine diameters of $M$. The observation made in the preceding paragraph suggests to consider the following function of a pair of oriented affine diameters
\begin{equation} \label{newfunct}
\omega(z_1-z_1^*,z_2-z_2^*) = \omega(z_1,z_2) + \omega(z_2,z_1^*) + \omega (z_1^*,z_2^*) + \omega (z_2^*,z_1).
\end{equation}
We shall show that the critical points of this function are 4-periodic orbits of the symplectic billiard. We start with a technical statement.

\begin{lemma} \label{affdiff}
The map $\varphi: M \to S^{2n-1}$ to the unit sphere, given by the formula
$$
z \mapsto \frac{z-z^*}{|z-z^*|},
$$
is a diffeomorphism. 
\end{lemma}

\begin{proof}One has the following characterization of affine diameters: a chord of a convex body is its affine diameter if and only if it is a longest chord of the body in a given direction, see \cite{So}. Since $M$ is strictly convex, for every direction $v \in S^{2n-1}$, there is a unique affine diameter $z z^*$. The map $v \mapsto z$ is inverse of the map $\varphi$, and this map is smooth. Thus the smooth map $\varphi$ has a smooth inverse map, that is, $\varphi$ is a diffeomorphism.
\end{proof}





Next we consider critical points of the function (\ref{newfunct}).

\begin{lemma} \label{crit}
The critical points of the function $\omega(z_1-z_1^*,z_2-z_2^*)$ are 4-periodic orbits of the symplectic billiard.
\end{lemma}

\begin{proof}Let $z\in M$ and let $v\in T_z M$ be a tangent vector. Denote by $v^* \in T_{z^*} M$ the image of $v$ under the differential of the involution $z \mapsto z^*$.

Let $(z_1,z_2)$ be a critical point of the function $\omega(z_1-z_1^*,z_2-z_2^*)$. Then, for every $v\in T_{z_1} M$, one has $\omega(v-v^*,z_2-z_2^*)=0$. 

Note that $\psi$, composed with normalization to unit vectors, is the map $\varphi$ of Lemma \ref{affdiff}, that is, a diffeomorphism. Hence the map $d\psi(z)v=v-v^*:T_zM\to\R^{2n}$ is an injection.


It follows that the vector $z_2-z_2^*$ is symplectically orthogonal to the hyperplane $T_{z_1} M$, that is, $z_2-z_2^* \in R(z_1)=R(z_1^*)$. The same argument shows that $z_1-z_1^* \in R(z_2)=R(z_2^*)$. Hence $z_1 z_2 z_1^* z_2^*$ is an orbit of the symplectic billiard in $M$.
\end{proof}

Let ${\mathcal D}$ be the set of pairs of oriented affine diameters of $M$. One has an action of the group $\Z_4$ on this set:
$$
(D_1,D_2) \mapsto (D_2,-D_1) \mapsto (-D_1,-D_2) \mapsto (-D_2, D_1) \mapsto (D_1,D_2).
$$
Let $F(D_1,D_2)$ be the function (\ref{newfunct}). This function is $\Z_4$-invariant. 

Let $U \subset {\mathcal D}$ be the manifold with boundary given by the inequality $F(D_1,D_2) \ge \eps$ for a sufficiently small generic positive $\eps$. The gradient of the function $F$ has the inward direction on the boundary of $U$, therefore the usual Morse-Lusternik-Schnirelman inequalities for the number of critical points apply. 

Note that if $(D_1,D_2) \in U$, then $D_1 \neq \pm D_2$, and the action of $\Z_4$ on $U$ is free. 
We need to describe  the topology of the quotient space $U/\Z_4$. We claim that it is homotopically equivalent to the lens space $L = S^{2n-1}/\Z_4$, where $\Z_4$ acts on the unit sphere in $\C^n$ by
$(z_1,\ldots,z_n) \mapsto (iz_1,\ldots,iz_n)$.

Let $V$ be the set of pairs of unit vectors $(e_1,e_2)$ with $\omega(e_1,e_2) >0$. Using Lemma \ref{affdiff}, we normalize $D_1$ and $D_2$ to unit vectors. Thus, at the first step, we get a $\Z_4$-equivariant retraction of $U$ to $V$.

Next we want to retract $V$ to the set of complex 2-frames $(e_1,e_2)$ satisfying $e_2=J e_1$. To this end, consider the function $\omega(e_1,e_2)$ on $V$. 

We claim that the critical points of this function are complex  frames with $e_2= J e_1$. Indeed, let $(e_1,e_2)$ be a critical point. Then, for every $v \in T_{e_1} S^{2n-1}$, one has $\omega(v,e_2) =0$. Hence $e_2 = \pm J e_1$. The condition $\omega(e_1,e_2) > 0$ excludes the minus sign. 

Thus the gradient of the function $\omega(e_1,e_2)$ retracts $V$   to the set of complex 2-frames. This set is $S^{2n-1}$ with the action of $\Z_4$ given by $J$.

It remains to use the Lusternik-Schnirelman lower bound for the number of critical points given by the category of the lens space $L$. It is known that ${\rm cat}(L)=2n$ (Kransnoselski  \cite{Kr}; alternatively, one can use the 2-fold covering $\RP^{2n-1} \to L$ that implies
$
{\rm cat}(L) \ge {\rm cat} (\RP^{2n-1}) = 2n,
$
where the inequality follows from the homotopy lifting property, see \cite{Ja}). 
This completes the proof of the theorem.
\end{proof}

\end{document}